\documentclass[12pt]{amsart}

\usepackage{amssymb,amsmath,xcolor, amsrefs, fullpage}
\usepackage{amsthm}
\usepackage{url}
\usepackage{hyperref}
\usepackage{tikz,subcaption}
\usepackage[enableskew]{youngtab}
\usepackage{ytableau,varwidth}

\definecolor{red}{rgb}{1,0,0}
\definecolor{blue}{rgb}{.2,.2,.8}
\definecolor{magenta}{rgb}{1,0,1}
\definecolor{dartmouthgreen}{rgb}{0.05, 0.5, 0.06}
\definecolor{purple(x11)}{rgb}{0.63,0.36,0.94}
\definecolor{turquoise}{rgb}{0.25, 0.87, 0.81}

\def\B{\mathcal B}
\def\A{\mathcal A}

\newcommand{\rank}{\operatorname{rank}}

\newtheorem{theorem}{Theorem}[section]
\newtheorem{corollary}[theorem]{Corollary}
\newtheorem{proposition}[theorem]{Proposition}

\newtheorem{lemma}[theorem]{Lemma}

\theoremstyle{definition}

\newtheorem{example}{Example}
\newtheorem{remark}{Remark}

\begin{document}

\title{Mock theta functions and related combinatorics}
\author{Cristina Ballantine, Hannah Burson, Amanda Folsom, \\ Chi-Yun Hsu, Isabella Negrini, and Boya Wen}

\address{Department of Mathematics and Computer Science, College of the Holy Cross, Worcester, MA 01610 USA} 
\email{cballant@holycross.edu}    
 
\address{School of Mathematics, University of Minnesota, Twin Cities, 
127 Vincent Hall 206 Church St. SE, Minneapolis, MN 55455, USA}  
\email{hburson@umn.edu}

\address{Department of Mathematics and Statistics, Amherst College, Amherst, MA 01002, USA}  
\email{afolsom@amherst.edu}  
 
\address{Department of Mathematics, University of California, Los Angeles, Math Sciences Building, 520 Portola Plaza, Box 951555, 
Los Angeles, CA 90095, USA}  
\email{cyhsu@math.ucla.edu}
\curraddr{Laboratoire Paul Painlev\'e, Universit\'e de Lille, 59000 Lille, France}
\email{chiyun.hsu@univ-lille.fr}

\address{Department of Mathematics and Statistics, McGill University, 805 Sherbrooke ST west, Montreal, Quebec H3A 2K6, Canada} 
\email{isabella.negrini@mail.mcgill.ca}
\curraddr{Simons Laufer Mathematical Sciences Institute (formerly MSRI), 17 Gauss Way,
Berkeley, California, 94720-5070, USA}

\address{Department of Mathematics, University of Wisconsin-Madison, 480 Lincoln Drive, Madison, WI 53706, USA} 
\email{bwen25@wisc.edu}
\curraddr{Simons Laufer Mathematical Sciences Institute (formerly MSRI), 17 Gauss Way,
Berkeley, California, 94720-5070, USA}

	\maketitle
	
\begin{abstract}
In this paper we add to the literature on the  combinatorial nature  of the mock theta functions, a collection of curious  $q$-hypergeometric series introduced by Ramanujan in his last letter to Hardy in 1920, which we now know to be important examples of mock modular forms.    Our work is inspired  by Beck's conjecture, now a theorem of Andrews, related to Euler's identity:  the excess of the number of parts in all partitions of $n$ into odd parts over the number of partitions of $n$ into distinct parts is equal to the number of partitions with only one (possibly repeated) even part and all other parts odd.  
We establish Beck-type identities associated to partition identities due to Andrews, Dixit, and Yee for the third order mock theta functions $\omega(q), \nu(q)$, and $\phi(q)$.
Our proofs are both analytic and combinatorial in nature, and involve mock theta generating functions and combinatorial bijections.
\end{abstract}

	\section{Introduction} 
\label{sec_mockintro}
\subsection*{Mock theta functions}
In Ramanujan's last letter to Hardy from 1920, he presented his \emph{mock theta functions,} a collection of 17 curious $q$-hypergeometric series including 
\begin{align*}
    \omega(q) &:= \sum_{k=0}^\infty \frac{q^{2k(k+1)}}{(q;q^2)^2_{k+1}}, \\
    \nu(q) &:= \sum_{k=0}^\infty \frac{q^{k(k+1)}}{(-q;q^2)_{k+1}},\\
    \phi(q) &:= \sum_{k=0}^\infty \frac{q^{k^2}}{(-q^2;q^2)_k},
\end{align*}
of the \emph{third order}.   Here and throughout, the $q$-Pochhammer symbol is defined for $n\in \mathbb N_0 \cup\{\infty\}$ by   
\begin{align*}
 (a;q)_n := \prod_{j=0}^{n-1}(1-a q^j) = (1-a)(1-aq)(1-aq^2) \cdots (1-aq^{n-1})
\end{align*}  and we assume $|q|<1$, so that all series converge absolutely.
Ramanujan didn't define what he meant by the order of a mock theta function, nor did he precisely define a mock theta function. However, we have since been able to extract a definition  from his own writing \cite{BerndtRankin} (see also the recent works \cite{GRO, Rhoadesmock}):\smallskip
\begin{quote}\emph{``Suppose there is a function in the Eulerian form and suppose that all or an infinity of points $q=e^{2i \pi m/n}$ are exponential
singularities and also suppose that at these points the asymptotic form of the function closes neatly\dots The question is: is the function taken the sum of two functions one of which is an ordinary theta function and the other
a (trivial) function which is $O(1)$ at all the points $e^{2 i \pi m/n}$? The answer is it is not necessarily so.
When it is not so I call the function Mock $\theta$-function. I have not proved rigorously that it is not necessarily so. But I have constructed
a number of examples...''}
\end{quote}
Ramanujan's reference to theta functions,   a class of modular forms, and  Eulerian forms, which are $q$-series expressible in terms of   \emph{$q$-hypergeometric} series (\cite{Fine, GasperRahman}) and similar in  shape  to $\omega(q), \nu(q),$ and $\phi(q)$, indirectly points back to earlier examples of Eulerian modular forms.  For example, Dedekind's $\eta$-function is an important  modular theta function of weight $1/2$ which can be expressed in terms of a $q$-hypergeometric series as follows:
\begin{align} \label{def_etahyp} q^{\frac{1}{24}}\eta^{-1}(\tau) = \sum_{n=0}^\infty \frac{q^{n^2}}{(q;q)_n^2},\end{align} {where} $q=e^{2\pi i \tau}$ {is} the usual modular variable, with $\tau$ in the upper half complex plane.  Ramanujan's letter on his mock theta functions claimed that mock theta functions behave like (weakly holomorphic) modular forms near roots of unity but are not themselves modular, hence the adjective \emph{mock}.   

The precise roles played by the mock theta functions within the theory of modular forms remained unclear in the decades following Ramanujan's death shortly after he wrote his last letter to Hardy.  However, the  importance of these functions was clear --  they have been shown to play meaningful roles in the diverse subjects of combinatorics, $q$-hypergeometric series,   mathematical physics,   elliptic curves and traces of singular moduli,    Moonshine and representation theory, and more.   Within the last 20 years we have also finally understood,  thanks to key work by Zwegers, Bruinier--Funke, and others including Bringmann--Ono and Zagier \cite{BFOR}, that the mock theta functions turn out to be examples of \emph{mock modular forms}, which are holomorphic parts of \emph{harmonic Maass forms}, modern relatives to ordinary Maass forms and modular forms.  This context has also allowed us to make more sense of the notion of the order of a mock theta function. For more background and information on these aspects of the mock theta functions, see, e.g., \cite{BFOR, DukeMock, FPerspectives, ZagierBBQ}.  

Turning to the first application of mock theta functions mentioned above, combinatorics, we recall that Dedekind's modular $\eta$-function may also be viewed as the reciprocal of the generating function for integer partitions.  That is, \eqref{def_etahyp} may also be written as 
\begin{equation}
\label{eqn:part}
    \prod_{n=1}^\infty \frac{1}{1-q^n} = \sum_{n=0}^\infty p(n) q^n = 1+q+2q^2+3q^3 + 5q^4 + 7q^5 + \cdots,
\end{equation}
 where $p(n)$ is the number of partitions of $n$.   That \eqref{eqn:part} is simultaneously a modular form and a combinatorial generating function has led to some deep and important results and theory.  Namely,    Hardy--Ramanujan  introduced their famous \emph{Circle Method} in analytic number theory, which combined with the modularity of Dedekind's   $\eta$-function, led to the following {exact} formula for the partition numbers \cite{Rad} 
$$
p(n)= 2 \pi (24n-1)^{-\frac{3}{4}} \sum_{k =1}^{\infty}
\frac{A_k(n)}{k}  I_{\frac{3}{2}}\left( \frac{\pi
\sqrt{24n-1}}{6k}\right),
$$
an infinite sum in terms of Kloosterman sums $A_k$ and Bessel functions $I_{s}$.

Like the modular $\eta$-function,  the mock theta functions may also be viewed as combinatorial generating functions.  For example, we have 
\begin{align*}
  q\omega(q) = \sum_{n=1}^\infty  a_\omega(n) q^n, \ \ \ \ \ \ \ 
  \nu(-q) = \sum_{n=0}^\infty a_\nu(n) q^n, \ \ \ \ \ \ \ 
  \phi(q) = \sum_{n=0}^\infty a_\phi(n) q^n, 
\end{align*}
where  $a_\omega(n)$
counts the number of partitions of $n$ whose parts, except for one instance of the largest part, form pairs of consecutive non-negative integers \cite[(26.84)]{Fine}; $a_\nu(n)$ counts the number of partitions of $n$ whose even parts are distinct, and if $m$ occurs as a part, then so does every positive even number less than $m$; and
$a_\phi(n) :=  sc_e(n) - sc_o(n),$ where $sc_{o/e}(n)$ counts the number of self-conjugate partitions $\lambda$ of $n$ with $L(\lambda)$ odd/even.  Here, $L(\lambda)$ is the number of parts of $\lambda$ minus the side length of its Durfee square.  {(See, e.g., \cite{AndrewsEncy} and Section~\ref{sec:notation} for more background on integer partitions.)}  

Using the newer theory of {mock modular forms}, we have results analogous to the celebrated Hardy--Ramanujan--Rademacher exact formula for $p(n)$; for example, due to Garthwaite \cite{Garthwaite} we have 
$$ a_\omega(n) =  \frac{\pi}{2\sqrt{2}}(3n+2)^{-\frac14} \mathop{\sum_{k=1}^\infty}_{(k,2)=1} \frac{(-1)^{\frac{k-1}{2}} A_k(\frac{n(k+1)}{2}-\frac{3(k^2-1)}{8})}{k}I_{\frac12}\left(\frac{\pi \sqrt{3n+2}}{3k}\right).$$
Numerous other papers, some of which we discuss in the sections that follow, have established further meaningful combinatorial results pertaining to the mock theta functions, including congruence properties, asymptotic properties, and more, adding to   broader and older theories which rest at the  intersection of  combinatorics and  modular forms.  

\subsection*{Beck-type partition identities}
In this paper we seek to add to the growing literature on understanding the combinatorial nature of the mock theta functions. More specifically, we study the total number of parts in certain sets of  partitions of $n$ related to  the third order mock theta functions $\omega(q), \nu(q),$ and $\phi(q)$.    In general, identities on the number of parts in all partitions of $n$ of a certain type have  been of  interest in the literature, dating back to work of Beck and Andrews.   Their work was motivated by Euler's famous partition identity, which states that for any positive integer $n$, $$p(n \ | \text{ odd parts}) = p(n \ | \ \text{distinct parts}),$$ and which may be immediately deduced from the identity 
$$\prod_{n=1}^\infty \frac{1}{1-q^{2n-1}} = \prod_{n=1}^\infty (1+q^n)$$ upon realizing that 
the ``modular" products appearing are generating functions for the partition functions in Euler's identity.  Here and throughout, we use the common notation
$p(n \ | \ \textit{conditions })$ to denote the number of partitions of $n$ subject to the given conditions.  For example, $p(n \ | \ \text{odd parts})$ equals the number of partitions of $n$ with odd parts.

While {the natural number-of-parts refinement} of Euler's identity is not true, namely  the number of partitions  of $n$ into exactly $m$ odd parts is not in general equal to the number of partitions   of $n$ into exactly $m$ distinct parts,   Beck conjectured and Andrews proved \cite{A17}   that the excess in the number of parts in all partitions of $n$ into odd parts over {the number of parts in} all partitions of $n$ into distinct parts is equal to the number of partitions with only one (possibly repeated) even part and all other parts odd. Andrews additionally showed that this excess is also equal to the number of partitions of $n$ with only one repeated part and all other parts distinct.   Andrews provided an analytic proof of this theorem using generating functions, and   Yang \cite{Yang19} and Ballantine--Bielak \cite{BB19} later  independently provided combinatorial proofs. 

Since Beck made the first conjecture of this type, combinatorial identities on the excess between the number of parts in all partitions of $n$ arising from a partition identity like Euler's are now fairly commonly referred to as  ``Beck-type identities." In the recent past, a number of other interesting
 Beck-type  companions to other important identities have been established -- see, e.g., \cite{AB19}, \cite{WIN5Lehmer},  \cite{BW21}, \cite{LW20}, \cite{Yang19}.  
 
 Here, we establish Beck-type identities associated to the third order mock theta functions $\omega(q), \nu(q),$ and $\phi(q)$ in Theorem \ref{T1}, Theorem \ref{thm:nu_partition}, and Theorem \ref{thm_phimain}, respectively. 
Our results may be viewed as Beck-type companion identities to partition identities for the third order mock theta functions $\omega(q), \nu(q),$ and $\phi(q)$ due to Andrews, Dixit and Yee in \cite{ADY}.   We devote Section \ref{sec:notation} to preliminaries on partitions, and state and prove our main results on $\omega(q), \nu(q),$ and $\phi(q)$ in Section \ref{BeckIdentity}, Section \ref{sec_nusection}, and Section \ref{sec_phi}, respectively. {As a Corollary to our main results, we also establish mock theta pentagonal-number-theorem-type results in Theorem \ref{pent} and Corollary \ref{cor_pent}.}   Generally speaking, our proofs are both analytic and combinatorial in nature, and involve mock theta generating functions and combinatorial bijections.

\section{Preliminaries on partitions} \label{sec:notation} 
Let $n\in\mathbb N_0$. A \emph{partition} of $n$, denoted  $\lambda=(\lambda_1, \lambda_2, \ldots, \lambda_j)$,     is a non-increasing sequence of positive integers $\lambda_1\geq \lambda_2 \geq \cdots \geq \lambda_j$ called \emph{parts} that add up to $n$. We refer to $n$ as the \emph{size} of $\lambda$.  The \emph{length} of $\lambda$ is the number of parts of $\lambda$, denoted by $\ell(\lambda)$.  We denote by $\ell_o(\lambda)$ and  $\ell_e(\lambda)$ the number of odd, respectively even parts of $\lambda$.  For convenience, we abuse notation and use $\lambda$ to denote either the multiset of its parts or the non-increasing sequence of parts.  We write $a\in \lambda$ to mean the positive integer $a$ is a part of $\lambda$. As mentioned in the introduction, we denote by $p(n)$ the number of partitions of $n$. The empty partition is the only partition of size $0$, thus, $p(0)=1$. We write $|\lambda|$ for the size of $\lambda$ and  $\lambda\vdash n$  to mean that $\lambda$ is a partition of size $n$. For a pair of partitions $(\lambda, \mu)$ we also write $(\lambda, \mu)\vdash n$ to mean $|\lambda|+|\mu|=n$. 
We use the convention that $\lambda_k=0$ for all $k>\ell(\lambda)$.  When convenient we will also use the exponential notation for parts in a partition:  the exponent of a part is the multiplicity of the part in the partition.  This notation will be used mostly for rectangular partitions.  We write $(a^b)$ for the partition consisting of $b$ parts equal to $a$. 
Further, we denote by calligraphy style capital letters the set of partitions enumerated by the function denoted by the same letter. For example, we denote by $q_o(n)$ the number of partitions of $n$ into distinct odd parts and by $\mathcal Q_o(n)$ the set of partitions of $n$ into distinct odd parts. {Moreover, when the size of the partitions is not explicitly stated in the notation of a set, we mean the set of all partitions with the properties implied by the notation. For example, $\mathcal Q_o=\bigcup_{n\geq 0}\mathcal Q_o(n)$.}

The \emph{Ferrers diagram} of a partition $\lambda=(\lambda_1, \lambda_2, \ldots, \lambda_j)$ is an array of left justified boxes such that the $i$th row from the top contains $\lambda_i$ boxes. In the literature, these are also referred to as \emph{Young diagrams}. We abuse notation and use $\lambda$ to mean a partition or its Ferrers diagram. The $2$-\emph{modular Ferrers diagram} of $\lambda$ is a Ferrers diagram in which row $i$ has $\lceil \frac{\lambda_i}{2}\rceil$ boxes, all but the first filled with $2$. The first box of row $i$ is filled with $2$, respectively $1$,  if $\lambda_i$ is even, respectively odd. \begin{example} The Ferrers diagram and the $2$-modular Ferrers diagram of  $\lambda=(5,4, 3, 3, 2,2)$ are shown in Figure \ref{fig:1}. \begin{figure}[h]
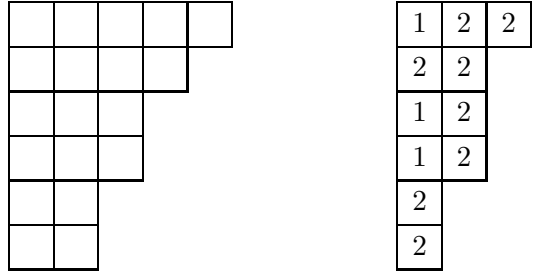
  $$\small\ydiagram{5, 4, 3, 3,2, 2} \ \ \ \ \ \ \ \ \ \ \ \ \ \ \ \ \ \small\begin{ytableau} 1& 2& 2\\ 2 & 2\\ 1& 2\\ 1& 2\\ 2\\ 2\end{ytableau}$$ 
\caption{A Ferrers diagram and 2-modular Ferrers diagram}
  \label{fig:1}
\end{figure} \end{example}

{Given a partition $\lambda$, its \emph{conjugate} $\lambda'$ is the partition for which the rows in its Ferrers diagram are precisely the columns in the Ferrers diagram of $\lambda$. For example, the conjugate of $\lambda=(5,4, 3, 3, 2,2)$ is $\lambda'=(6,6,4,2,1)$. A partition  is called \emph{self-conjugate} if it is equal to its conjugate.}

The Durfee square of a partition $\lambda$ is the largest square that fits inside the Ferrers diagram of $\lambda$, i.e., the partition $(a^a)$, where $a$ is such that  $\lambda_a\geq a$ and $\lambda_{a+1}\leq a$. For example, the Durfee square of $\lambda=(5,4, 3, 3, 2,2)$  is $(3^3)=(3,3,3)$.

For more details on partitions, we refer the reader to \cite{AndrewsEncy}.
 
An \emph{odd Ferrers diagram} $F$ is a Ferrers diagram  filled with $1$ and $2$ such that the first row is filled with $1$ and the remaining rows form the $2$-modular Ferrers diagram of a partition $\lambda$ with all parts odd. If the first row has length $k$, we identify the odd Ferrers diagram $F$ with  the pair $(k,\lambda)$. The \emph{size} of an odd Ferrers diagram ${F}$ is the sum of all entries in the boxes of diagram and is denoted by $|{F}|$. %The length of $F$ is the number of rows in the diagram. 

\begin{example}
\label{eg_odd_ferrers}
Figure \ref{fig:2} shows the odd Ferrers diagram of size $44$ with  $7$ rows corresponding to the pair $(k,\lambda)$ with $k=8$ and $\lambda=(11,7,7,5,5,1)$. \begin{figure}[h]
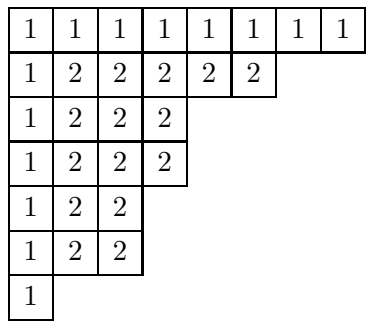

\begin{center}
\small \begin{ytableau}
1&1&1&1&1&1& 1&1\\
1&2&2&2&2&2\\
1&2&2&2\\
1&2&2&2\\
1&2&2\\
1&2&2\\
1
\end{ytableau}
\end{center} 
\caption{An odd Ferrers diagram}
  \label{fig:2}
\end{figure}
\end{example}

The rank of a partition $\lambda$, denoted  $r(\lambda)$, is defined as $r(\lambda)=\lambda_1-\ell(\lambda)$, or equivalently, the number of columns minus the number of rows in its Ferrers diagram.   In \cite{BG02}, the $M_2$-rank of a partition is defined as the number of columns minus the number of rows in its $2$-modular diagram. 
The rank of an odd Ferrers diagram $F=(k, \lambda)$, denoted $\rank (F)$, is defined as the number of columns minus the number of rows of $F$, or equivalently, $\rank(F) = k-\ell(\lambda)-1$. 

\section{The mock theta function $\omega$}\label{BeckIdentity}
{Recall from Section \ref{sec_mockintro} that} Ramanujan's third order mock theta function $\omega$ is defined by 
$$\omega(q):= \sum_{k=0}^\infty \frac{q^{2k(k+1)}}{(q;q^2)_{k+1}^2}.$$   It is known \cite[(26.84)]{Fine} that  
$$q\omega(q) = A_\omega(q) := \sum_{k=1}^\infty \frac{q^{k}}{(q;q^2)_{k}} = \sum_{n=1}^\infty a_\omega(n) q^n,$$ where $a_\omega(n)$ 
counts the number of partitions of $n$ whose parts, except for one instance of the largest part, form pairs of consecutive non-negative integers. {We are allowing pairs of consecutive integers to be $(0,1)$, but we are not considering $0$ as a part of the partition.}
There is also the (highly non-trivial) identity by Andrews--Dixit--Yee \cite{ADY}:  
$$q\omega(q) = B_\omega(q) := \sum_{k=1}^\infty \frac{q^k}{(q^k;q)_{k+1}(q^{2k+2};q^2)_\infty} = \sum_{n=1}^\infty b_\omega(n)q^n, $$ 
where $b_\omega(n)$ counts the number of partitions of $n$ such that all odd parts are less than twice the smallest part.  Hence, $a_\omega(n) = b_\omega(n)$. 

We define two variable generalizations of $A_\omega(q)$ and $B_\omega(q)$  as follows.
Let 
\begin{align}\label{def_Azq}A_\omega(z;q):=\sum_{k=1}^\infty \dfrac{zq^{k}}{(1-zq)(z^2q^3;q^2)_{k-1}}=\sum_{m=1}^\infty\sum_{n=1}^{\infty}a_{\omega}(m,n)z^mq^n,\end{align}
where $a_\omega(m,n)$ counts the number of partitions of $n$ with $m$ parts, which except for one instance of the largest part, form pairs of consecutive non-negative integers. {To see this, one can re-write the   denominator of the $k$th summand as
$$(1-z q^{0+1})(1-z^2 q^{1+2})(1-z^2 q^{2+3})\cdots (1-z^2 q^{(k-1)+k}),$$ and use the same convention as noted above for the combinatorial interpretation of $a_\omega(n)$ for which  $(0,1)$ is an allowed pair of consecutive non-negative integers but for which $0$ is not considered a part of the partition.} Let
\begin{align}\label{eqn_Bzq}
B_\omega(z;q):= \sum_{k=1}^\infty \frac{zq^k}{(zq^k;q)_{k+1}(zq^{2k+2};q^2)_\infty}=\sum_{m=1}^{\infty}\sum_{n=1}^{\infty}b_{\omega}(m,n)z^mq^n,
\end{align}
where $b_{\omega}(m,n)$ counts the number of partitions of $n$ with $m$ parts, whose odd parts are less than twice the smallest part.  {In particular, we have $A_\omega(1;q) = A_\omega(q)$, and $B_\omega(1;q) = B_\omega(q)$.}

Following the notation convention introduced in Section \ref{sec:notation}, $\A_\omega(n)$ is the set of partitions of $n$ whose parts, except for one instance of the largest part, form pairs of consecutive non-negative integers.
We denote by $\A_{\omega,2}(n)$ the set of odd Ferrers diagrams of size $n$, and then  $a_{\omega,2}(n)=|\A_{\omega,2}(n)|$.   

{We next define two   generating functions, $A_{\omega,2}$ and $\widetilde{A}_{\omega,2}$, for odd Ferrers diagrams, which we later show are related to $A_\omega$ and $B_\omega$.  Namely, we let}
 \begin{align}\label{def_A2zq}A_{\omega,2}(z;q) := \sum_{k=1}^\infty \frac{zq^{k}}{(zq;q^2)_{k}} = \sum_{m=1}^\infty\sum_{n=1}^\infty a_{\omega,2}(m,n)z^m q^n,\end{align} where $a_{\omega,2}(m,n)$ counts the number of odd Ferrers diagrams of size $n$ with $m$ rows. We note that this interpretation was introduced by Andrews in \cite{A18}.
We also let 
\begin{align}\label{def_tA2zq}\widetilde A_{\omega,2}(z;q) := \sum_{k=1}^\infty \frac{z^kq^{k}}{(q;q^2)_{k}} = \sum_{m=1}^\infty\sum_{n=1}^\infty \widetilde a_{\omega,2}(m,n)z^m q^n,\end{align} where $\tilde a_{\omega,2}(m,n)$ counts the number of odd Ferrers diagrams of size $n$ with $m$ columns.
The combinatorial interpretation of $\widetilde{A}_{\omega,2}(z;q)$ was first described by Li and Yang in \cite[(2.22)]{LY19}. 

\begin{lemma} \label{lem_AandA2}
There is an explicit bijection $\A_\omega(n) \xrightarrow{\sim} \A_{\omega,2}(n)$.
Moreover, if $\mu\mapsto {F}_\mu$ under this bijection, then the number of parts of $\mu$ is equal to the number of rows of ${F}_\mu$ plus the number of rows of ${F}_\mu$ containing at least one $2$, i.e., 
\begin{align} \label{AA2}
    A_\omega(z;q) = A_{\omega,2}(z^2;q)\cdot \frac{1-z^2q}{z(1-zq)}.
\end{align}

\end{lemma} 
\begin{proof}
Start with $\mu\in  \A_\omega(n)$, remove one instance of the largest part $\mu_1$, and merge the (consecutive) pairs of parts of the remaining partition to obtain a partition $\lambda$ into odd parts. Then the corresponding odd Ferrers diagram is ${F}_\mu=(\mu_1, \lambda)$. 
 This transformation is invertible: given $(k,\lambda)\in \A_{\omega,2}(n)$, each part of $\lambda$ is odd and hence the sum of a pair of consecutive non-negative integers.
 %The corresponding partition then consists of $k$ and the consecutive pairs of integers.  
 The corresponding partition has parts $k$ and all pairs of parts obtained by splitting the parts of $\lambda$ into consecutive integers.
 The connection between the number of parts of $\mu$ and the number of rows of ${F}_\mu$ is clear from this explicit bijection.
\end{proof}

\begin{theorem} \label{T1} The excess of the number of parts in all partitions in $ \A_\omega(n)$ over the  number of parts in all partitions in $\B_\omega(n)$ equals the number of rows containing at least one $2$ in all odd Ferrers diagrams $F=(k,\lambda)$ of size $n$, which is the same as the number of parts greater than $1$ in $\lambda$.
\end{theorem} 

We will provide four proofs of this theorem. 

We first introduce some useful identities.
From equation (8) of \cite{AY}, we have 
\begin{equation}\label{BzA2}
    B_\omega(z;q)=\widetilde A_{\omega,2}(z;q).
\end{equation}
Moreover, from equation (16) of \cite{AY}, we have 
\begin{equation}\label{nozprod}
    \widetilde A_{\omega,2}(z;q)=A_{\omega,2}(z;q).
\end{equation}
(This can also be seen from the fact that conjugation provides a bijection $\widetilde\A_{\omega,2}(m,n) \xrightarrow{\sim} \A_{\omega,2}(m,n)$ 
\cite[p.539]{LY19}.)
From \eqref{BzA2} and \eqref{nozprod}, we have \begin{align}\label{eqn_BzA2} 
B_\omega(z;q)  = A_{\omega,2}(z;q).
\end{align}

 By Lemma \ref{lem_AandA2} and Theorem \ref{T1}, or, after differentiating \eqref{eqn_BzA2} at $z=1$, we have the following result.
 
\begin{corollary} The total number of parts in all partitions in  $\B_\omega(n)$ equals the total number of rows in all odd Ferrers diagrams of size $n$. 
\end{corollary}

%\begin{proof}%[Second proof]
%From \eqref{eqn_BzA2} we have 
%$$
%\left.\dfrac{\partial (A_{\omega,2}(z;q))}{\partial z} \right|_{z=1}=\left.\dfrac{\partial (B_\omega(z;q))}{\partial z} \right|_{z=1}.
%$$
%The left-hand side of this equality is the generating function for the total number of  rows in all odd Ferrers diagrams of size $n$, while the right-hand side is the generating function for the total number of parts in all partitions in $\B_\omega(n)$. The statement follows.
%\end{proof}

All four proofs  of Theorem \ref{T1}  make use of the fact that
\begin{align}\label{abdiff}
\left.\frac{\partial (A_\omega(z;q)- B_\omega(z;q)) }{\partial z}\right|_{z=1} \end{align} is the generating function for the excess of the number of parts in all partitions in $ \A_\omega(n)$ over the  number of parts in all partitions in $\B_\omega(n)$.

\begin{proof}[First proof]
We compute the derivative difference \eqref{abdiff}, using \eqref{AA2}, \eqref{BzA2}, and \eqref{nozprod}:
\begin{align*}
\left.\frac{\partial (A_\omega(z;q)- B_\omega(z;q)) }{\partial z}\right|_{z=1}
&=\left.\frac{\partial}{\partial z}\right|_{z=1} \left( \frac{1-z^2q}{z(1-zq)}\cdot \widetilde A_{\omega,2}(z^2;q)- \widetilde A_{\omega,2}(z;q)\right)  \\
&=\left.\frac{\partial \widetilde A_{\omega,2}(z;q)}{\partial z}\right|_{z=1} - \frac{1}{1-q}A_{\omega}(q) \\
&=\sum_{k=1}^\infty\frac{kq^{k}}{(q;q^2)_{k}}-\frac{1}{1-q}\sum_{k=1}^\infty\frac{q^{k}}{(q;q^2)_{k}}.
\end{align*}
The second term $\frac{1}{1-q}\sum_{k=1}^\infty\frac{q^{k}}{(q;q^2)_{k}}$ is the generating function for the number of pairs $({F},(1^b))\vdash n$, where ${F}$ is an odd Ferrers diagram and $b\geq0$ is an integer.

By mapping a pair $({F},(1^b))$ to an odd Ferrers diagram with at least $b$ rows of size $1$ and coloring the final $b$ rows of size $1$, we can see that  $\frac{1}{1-q}\sum_{k=1}^\infty\frac{q^{k}}{(q;q^2)_{k}}$ is 
also the generating function for the number of odd Ferrers diagrams ${F}=(k,\lambda)$ weighted by 
$m_\lambda(1) + 1$, where $m_\lambda(1)$ 
is the number of parts equal to  $1$ in $\lambda$.  

Hence $\left.\frac{\partial (A_\omega(z;q)- B_\omega(z;q)) }{\partial z}\right|_{z=1}$ is the generating function for the number of odd Ferrers diagrams ${F}=(k,\lambda)$ weighted by $k-(m_\lambda({1})+1)$.

Note that conjugation provides a bijection between odd Ferrers diagrams of size $n$ with $m$ rows and odd Ferrers diagrams of size $n$ with $m$ columns. Hence for a conjugate pair ${F}=(k,\lambda)$ and ${F'}=(j,\mu)$, we have 
\begin{align*}
    &k-(m_\lambda({1})+1)+j-(m_\mu({1})+1)=j-(m_\lambda({1})+1)+k-(m_\mu({1})+1)\\=&(\ell(\lambda)+1)-(m_\lambda({1})+1)+(\ell(\mu)+1)-(m_\mu({1})+1).
\end{align*}

 Therefore, summing over all odd Ferrers diagrams of size $n$, the generating function stays the same if we replace the weight by $(\ell(\lambda)+1)-(m_\lambda({1})+1)$, which is the number of rows containing at least one $2$ in ${F}$.
\end{proof}

\begin{proof}[Second proof] 
We compute the derivative difference \eqref{abdiff}, using \eqref{AA2} and \eqref{eqn_BzA2}:
\begin{align*}
\left.\frac{\partial (A_\omega(z;q)- B_\omega(z;q)) }{\partial z}\right|_{z=1}
&=\left.\frac{\partial}{\partial z}\right|_{z=1} \left( \frac{1-z^2q}{z(1-zq)}\cdot  A_{\omega,2}(z^2;q)-  A_{\omega,2}(z;q)\right)  \\
&=\left.\frac{\partial  A_{\omega,2}(z;q)}{\partial z}\right|_{z=1} - \frac{1}{1-q}A_{\omega}(q).
\end{align*}

We have seen in the first proof that the second term $\frac{1}{1-q}A_\omega(q)$ is the generating function for the number of odd Ferrers diagrams ${F}=(k,\lambda)$ weighted by $m_\lambda({1})+1$.
Hence $\left.\frac{\partial (A_\omega(z;q)- B_\omega(z;q)) }{\partial z}\right|_{z=1}$ is the generating function for the number of rows containing at least one $2$ in all odd Ferrers diagrams of size $n$.
\end{proof}

\begin{proof}[Third proof]
We compute the derivative difference \eqref{abdiff}, using \eqref{eqn_BzA2}:
\begin{align}
\left.\frac{\partial (A_\omega(z;q)- B_\omega(z;q)) }{\partial z}\right|_{z=1}
&=\left.\frac{\partial}{\partial z}\right|_{z=1} \left(  A_{\omega}(z;q)-  A_{\omega,2}(z;q)\right) \label{eqn_two_deriv_diff} \\
&=  \sum_{k=1}^\infty \frac{q^{k}}{(q;q^2)_{k}}\left(\sum_{j=1}^{k-1} \frac{q^{2j+1}}{1-q^{2j+1}}\right). 
\label{eqn_d(A-B)dz2}
\end{align}

This is the generating function for the number of pairs $({F}, ((2j+1)^b))\vdash n$, where ${F}$ is an odd Ferrers diagram and $j,b\geq1$ are integers. 
For each pair $({F}, ((2j+1)^b))\vdash n$, we insert $b$ copies of $(2j+1)$ as $2$-modular rows into ${F}$ and color the final $b$ rows of size $(2j+1)$ to obtain a colored odd Ferrers diagram. The number of such colored odd Ferrers diagrams of size $n$ is equal to the number of rows containing at least one $2$ in all odd Ferrers diagrams of size $n$.
\end{proof}

\begin{proof}[Fourth proof.] 

From (\ref{eqn_two_deriv_diff}), we have
\begin{equation}
    \sum_{m=1}^{\infty} (m a_{\omega}(m,n)-m b_{\omega}(m,n))=\sum_{m=1}^{\infty} (m a_{\omega}(m,n)-m a_{\omega,2}(m,n)),
\end{equation}
for each $n\geq 1$.  The left hand side of \eqref{eqn_two_deriv_diff} is the excess in the statement of Theorem \ref{T1}, whereas the right hand side is the excess of the number of parts in all partitions in $ \A_\omega(n)$ over the number of rows in all odd Ferrers diagrams in $\A_{\omega,2}(n)$. 
By Lemma \ref{lem_AandA2}, $\A_\omega(n)\xrightarrow{\sim} \A_{\omega,2}(n)$, and the excess is precisely the number of rows containing at least one $2$ in all odd Ferrers diagrams of size $n$.
\end{proof}

From the third proof of Theorem \ref{T1}, we obtain new interpretations of the derivative difference  \eqref{eqn_two_deriv_diff}  in Corollaries \ref{c1} and \ref{c2} below.   These are analogous to the original Beck identity which can be reinterpreted as follows. The excess of the total number of parts in all partitions of $n$ into distinct parts over the  total number of parts in all partitions of $n$ into odd parts equals the number of pairs $(\xi, \eta)\vdash n$, where $\xi$ is a partition into odd parts and $\eta$ is a rectangular partition into equal even parts.  This is also the number of pairs $(\xi, \eta)\vdash n$, where $\xi$ is a partition into distinct parts and $\eta$ is a rectangular partition with at least two parts. 

\begin{corollary} \label{c1} The excess of the total number of parts in all partitions in $\A_\omega(n)$ over the  total number of parts in all odd Ferrers diagrams of size $n$ equals the number of pairs $(\xi, \eta)\vdash n$ where, $\xi$ is an odd Ferrers diagram and $\eta$ is a rectangular partition into odd parts of size at least $3$.\end{corollary}

\begin{corollary}\label{c2}
The excess of the number of parts in all partitions in $\A_\omega(n)$ over the number of parts in all partitions in $\B_\omega(n)$ equals the number of pairs $(\lambda, \eta)\vdash n$, where $\lambda\in\A_\omega$ and $\eta$ is a rectangular partition into odd parts of size at least $3$.
\end{corollary}

\section{The mock theta function $\nu$}\label{sec_nusection} Recall from Section \ref{sec_mockintro} the mock theta function 
	$$\nu(q) := \sum_{k=0}^\infty \frac{q^{k(k+1)}}{(-q;q^2)_{k+1}}.$$  
	We write 
	\begin{align}\label{def_numinusq}\nu(-q) = A_\nu(q) = \sum_{n=0}^\infty a_\nu(n) q^n.\end{align}
	Since $k(k+1)=2+4+\cdots+2k$, $a_\nu(n)$ counts the number of partitions of $n$ whose even parts are distinct, and if $m$ occurs as a part, then so does every positive even number less than $m$. 

	 Let
		$$
		A_{\nu,2}(q) := \sum_{k=0}^\infty (-q;q^2)_kq^k =: \sum_{n=0}^\infty a_{\nu,2}(n) q^n.
		$$ We recall   \cite[(44)]{ADY} which gives the identity
$$ \sum_{m=0}^\infty \frac{q^{m^2}x^m}{(y;q^2)_{m+1}} = \sum_{m=0}^\infty (-xq/y;q^2)_{m} \ y^m.$$
  Letting $x=y=q$, we have 
		 $$\nu(-q) = A_{\nu,2}(q).$$
Note that $A_{\nu,2}(q)$ is the generating function for the number of  odd Ferrers diagrams $(k,\lambda)$ where the partition $\lambda$ has distinct parts. 
	
We also define $$B_\nu(q):=
	\sum_{k=0}^\infty q^k (-q^{k+1};q)_k (-q^{2k+2};q^2)_\infty = \sum_{n=0}^\infty b_\nu(n) q^n,$$ where $b_\nu(n)$ counts the number of partitions of $n$ into distinct parts, in which each odd part is less than twice the smallest part, and zero can be a part (note that this is different from our usual convention). For example, $(6,4,3,2)$ and $(6,4,3,2,0)$
 are counted as different partitions, the former from the term $q^k (-q^{k+1};q)_k (-q^{2k+2};q^2)_\infty$ for $k=2$ while the latter from the term for $k=0$.   Then, as stated in \cite[Theorem 4.1]{ADY}, we have the identity \begin{align}\label{eqn_nuqBnu}\nu(-q)=B_\nu(q).\end{align}

%Let $\A_{\nu}(n)$ denote the set of partitions of $n$ whose even parts are distinct, and if $m$ occurs as a part, then so does every positive even number less than $m$. 
%Let $\A_{\nu,2}(n)$ denote the set of odd Ferrers diagrams $(k,\lambda)$ of size $n$ where the partition $\lambda$ has distinct parts.
%Let $\B_{\nu}(n)$ denote the set of partitions of $n$ with distinct parts in which each odd part is less than twice the smallest part and zero can be a part.
%Thus $|\A_{\nu}(n)|=a_\nu(n)$, $|\A_{\nu,2}(n)|=a_{\nu,2}(n)$ and $|\B_{\nu}(n)|=b_\nu(n)$.%}

	\begin{lemma} \label{lem_AnuandAnu2}
 There is an explicit bijection  $ \A_{\nu,2}(n)\xrightarrow{\sim} \A_\nu(n) $. Moreover, if $(k,\lambda) \mapsto \pi$ under this bijection, then $\ell(\pi)=k$ and $\ell(\lambda)$ is the number of even parts in $\pi$. 
	\end{lemma} 
	\begin{proof}
	We adapt the bijection in \cite[Theorem 1.3]{LY19}.
		We start with an odd Ferrers diagram $F=(k,\lambda)\in \A_{\nu,2}(n)$, where $\lambda$ has distinct parts and length $\ell$.
   We will associate to $F$ a partition $\pi$ in $\A_{\nu}(n)$. 
   Consider the subdiagram  $T=(\ell, (2\ell-1, 2\ell-3, \ldots, 3,1))$  of $F$. 
%   \ccom{I think the second entry should be $2\ell-3$.} 
%   \cycom{Agree and changed}
   We map $T$ to the partition $\varepsilon=(2\ell, 2\ell-2, \ldots, 4, 2)$.  We remove $T$ from $F$ and shift all remaining boxes to the left to obtain a diagram $R$. The conjugate of $R$ is the $2$-modular diagram of a partition $\rho$ with odd parts. Define $\pi:=\varepsilon\cup \rho$.
	 
    From the procedure above, we see that $\pi$ has $k$ parts and the number of parts of $\lambda$ is equal to the number of even parts in $\pi$.
\end{proof}
	\begin{example}
		Let $F=(k,\lambda)=(5, (9,5,1))$. As in Lemma \ref{lem_AnuandAnu2}, we decompose $F=T+R$ as shown in Figure \ref{fig:3}. 
		\begin{figure}[h]
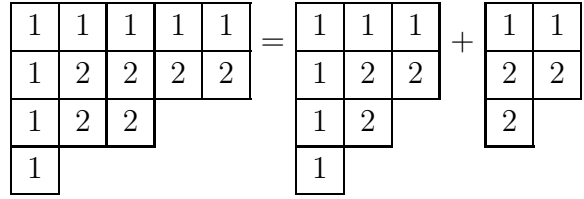

		\begin{center}
	\begin{ytableau}
    1&1&1&1&1\\
    1&2&2&2&2\\
    1&2&2\\
    1
    \end{ytableau}
    =
    \begin{ytableau}
    1&1&1\\
    1&2&2\\
    1&2\\
    1
    \end{ytableau}
    +
    \begin{ytableau}
    1& 1\\
    2& 2\\
    2
    \end{ytableau}
    \end{center} \caption{A decomposition of an odd Ferrers diagram as in Lemma \ref{lem_AnuandAnu2}}
  \label{fig:3}\end{figure}
    Then, $\varepsilon=(6,4,2)$,  $\rho=(5,3)$ and $\pi=(6,5,4,3,2)$. 
    \end{example}
	
{As in Section \ref{BeckIdentity}, we introduce two variable generalizations of $A_\nu(q), A_{\nu,2}(q),$ and $B_\nu(q)$ in which the exponent of $z$ keeps track of the number of parts in partitions.}

Let
	\begin{equation}
	\label{A_nu}
	    A_{\nu}(z;q):=\sum_{k=0}^\infty \frac{z^kq^{k^2+k}}{(zq;q^2)_{k+1}} =:\sum_{m=0}^\infty\sum_{n=0}^\infty a_{\nu}(m,n)z^mq^n.
	\end{equation}
	Since $z^kq^{k^2+k}=zq^2\cdot zq^4\cdots zq^{2k}$,
we find that
 $a_{\nu}(m,n)$ counts the number of partitions in $\A_{\nu}(n)$ with $m$ parts. %\ccom{I don't know about this change. To me it seems less clear now. Maybe we should explain after \ref{def_numinusq} that $q^{k(k+1}$ generates parts $2, 4, \ldots, k$, and $\frac{1}{(q;q^2)_\infty}$ generates odd parts at most $2k+1$. Then just saying that $z$ keeps track of the number of parts should be enough to give the meaning of $a_{\nu}(m,n)$.}

Let
	\begin{equation}
	\label{A_nu_two}
	    A_{\nu,2}(z;q):=\sum_{k=0}^\infty(-zq;q^2)_k z q^k =\sum_{m=0}^\infty\sum_{n=0}^\infty a_{\nu,2}(m,n)z^mq^n,
	\end{equation}
	where $a_{\nu,2}(m,n)$ counts the number of odd Ferrers diagrams $(k, \lambda)$ in $\A_{\nu,2}(n)$ with $m$ rows.
	
Let
\begin{equation}
     B_\nu(z;q):=\sum_{k=0}^\infty zq^k (-zq^{k+1};q)_k (-zq^{2k+2};q^2)_\infty= \sum_{m=0}^\infty\sum_{n=0}^\infty b_{\nu}(m,n)z^mq^n,
\end{equation}
where $ b_{\nu}(m,n)$ counts the number of partitions in $\B_{\nu}(n)$ with $m$ parts. Recall that partitions in $\B_{\nu}(n)$ can have $0$ as a part, which we do count in the number of parts -- for example, the partition $(6,4,2)$ has three parts, while $(6,4,2,0)$ has four parts.

\begin{theorem}
\label{thm:nu_partition}
The excess of the total number of parts in all partitions in $\A_{\nu}(n)$ over  the  total  number  of  parts  in all partitions in $\B_{\nu}(n)$ equals the sum of the number of odd parts minus $1$ over all partitions in $\A_{\nu}(n)$, or equivalently the sum of ranks over all odd Ferrers diagrams in $\A_{\nu,2}(n)$. {If $n\geq 1$, the excess is non-negative.}
\end{theorem}
 We provide two proofs of this theorem. 
\begin{proof}[Proof 1]  We have 
\begin{align*}
\dfrac{\partial}{\partial z}&\left.\left(A_{\nu}(z;q) - B_{\nu}(z;q)\right)\right|_{z=1} \\ 
 & = \dfrac{\partial}{\partial z} \left.\left(\sum_{k=0}^\infty \dfrac{z^kq^{k^2+k}}{(zq;q^2)_{k+1}} -\sum_{k=0}^\infty zq^k(-zq^{k+1};q)_k(-zq^{2k+2};q^2)_\infty\right)\right|_{z=1}\\
&=\dfrac{\partial}{\partial z}\left.\left(\sum_{k=0}^\infty \dfrac{z^kq^{k^2+k}}{(zq;q^2)_{k+1}}-\sum_{k=0}^\infty \frac{z^{k+1}q^{k^2+k}}{(q;q^2)_{k+1}}\right)\right|_{z=1}\\
&=\sum_{k=0}^\infty \dfrac{q^{k^2+k}}{(q;q^2)_{k+1}}\left(\left(\sum_{j=0}^k\frac{q^{2j+1}}{1-q^{2j+1}}\right)-1\right)\\
&=: \sum_{n=0}^\infty c(n) q^n,
\end{align*}
where we use \cite[(7)]{AY} in the second line above.  Note that $c(n)$ counts the number of odd parts in all partitions in $\A_{\nu}(n)$ minus the number of partitions in $\A_{\nu}(n)$. 
Rephrased, we have
$$\sum_{n=0}^\infty c(n)q^n=\sum_{n=0}^{\infty}\sum_{\pi\in \A_{\nu}(n)}(\ell_o(\pi)-1)q^{n}.$$

To show that $c(n)\geq 0$ for $n\geq 1$,  notice that the only partition $\pi$ with $\ell_o(\pi)-1<0$ is   $\pi=(2m,2m-2, \ldots, 2)$. For this $\pi$, we can split the largest  part $2m$ into two odd parts, namely $2m-1$ and $1$, to obtain $\widetilde{\pi}=(2m-1, 2m-2, \ldots, 2, 1)$, another partition of the same size in $\A_{\nu}$, such that $(\ell_o(\pi)-1)+(\ell_o(\widetilde{\pi})-1)=-1+1=0$. All the other partitions in $\A_{\nu}$ have at least one odd part. Therefore $c(n)=\sum_{\pi\in \A_{\nu}(n)}(\ell_o(\pi)-1)\geq 0$. 
\end{proof}

\begin{proof}[Proof 2]
From Lemma \ref{lem_AnuandAnu2}, if the partition $\pi\in \A_{\nu}(n)$ corresponds to the odd Ferrers diagram $F=(k,\lambda)\in \A_{\nu,2}(n)$, then the number of parts of $\pi$ is $k$.
In terms of generating functions, we  have the identity
\[
A_\nu(z;q) = \sum_{k=0}^{\infty} \frac{z^kq^{k^2+k}}{(zq;q^2)_{k+1}} = \sum_{k=0}^\infty (-q;q^2)_k z^k q^k.
\]
Thus, $\left.\frac{\partial}{\partial z} \right|_{z=1} A_\nu(z;q)$ is the generating function for the total number of columns in all odd Ferrers diagrams in $\A_{\nu,2}(n)$.

On the other hand, from \cite[Theorem 1]{AY} we have
\[
B_\nu(z;q) = \sum_{k=0}^\infty \frac{z^{k+1}q^{k^2+k}}{(q;q^2)_{k+1}}.
\]
Again from Lemma \ref{lem_AnuandAnu2}, if the partition $\pi\in \A_{\nu}(n)$ corresponds to the odd Ferrers diagram $F=(k,\lambda)\in \A_{\nu,2}(n)$, then the number of even parts in $\pi$ is equal to $\ell(\lambda)$.
In terms of generating functions, we have 
\[
\sum_{k=0}^\infty \frac{z^{k+1}q^{k^2+k}}{(q;q^2)_{k+1}} = \sum_{k=0}^\infty (-zq;q^2)zq^k = A_{\nu,2}(z;q).
\]
(This is also \cite[Theorem 2]{AY}.)
Hence $B_\nu(z;q) = A_{\nu,2}(z;q)$ and
so $\left.\frac{\partial}{\partial z} \right|_{z=1} B_\nu(z;q)$ is the generating function for the total number of rows in all odd Ferrers diagrams in $\A_{\nu,2}(n)$.

Combining these, we conclude that $\left.\frac{\partial}{\partial z} \right|_{z=1} (A_\nu(z;q)-B_\nu(z;q))$ is the generating function for the sum of ranks of all odd Ferrers diagrams in $\A_{\nu,2}(n)$.

Given an odd Ferrers diagram $F=(m,\lambda)\in \A_{\nu,2}(n)$, we have $m\geq \ell(\lambda)$ since the parts of $\lambda$ are distinct. Then, $\rank(F)=m-\ell(\lambda)-1\geq -1$ and $\rank(F)=-1$ if and only if $m=\ell(\lambda)$, in which case  $F=(m,(2m-1,2m-3,\ldots,1))$. Hence, there is at most one odd Ferrers diagram with rank $-1$ in $\A_{\nu,2}(n)$. If $\rank(F)=-1$,  the conjugate of $F$ is $F'=(m+1,(2m-1, 2m-3, \ldots, 3))\in \A_{\nu,2}(n)$,  $\rank(F')=(m+1)-m=1$, and thus $\rank (F)+\rank(F')=0$.  Since all  other odd Ferrers diagrams in $\A_{\nu,2}(n)$ have non-negative  rank, it follows that $c(n)\geq 0$.\end{proof}

% \acom{In what follows, we use the notation $A(n \ | \ \textit{conditions })$
% to denote the subset of partitions of $n$ from the set $A(n)$ subject to the given conditions.  The standard notation $|A|$ is used to denote the size of a set $A$.} \ccom{Or "... we use the notation $A(n \mid X)$
% to denote the subset of partitions of $n$ from the set $A(n)$ satisfying condition $X$. ..."}

% \ccom{I think this is the only place where we use this notation. It might make more sense to just rewrite the statement of the theorem. Before the theorem, we can write:

% "

{We end this section by investigating the parity of $c(n)$. To this end, we first prove a result similar to Euler's Pentagonal Number Theorem.} 
Let $a^e_\nu(n)$ (resp. $a^o_\nu(n)$) be the number of partitions in $\mathcal A_\nu$ with an even (resp. odd) number of parts.

%{\acom{need some language to introduce this Theorem/Corollary}}\incom{How about: "We will now deduce an interesting property of the coefficients of the derivative difference studied above. We will need at first the following theorem:".}

\begin{theorem} \label{pent}For any non-negative integer  $n$ we have 
$$a^e_\nu(n)= a^o_\nu(n) +e(n),$$ where $$e(n)=\begin{cases} 1,& \mbox{ if } n=3j^2+2j \mbox{ for some } j\geq 0,\\ -1, & \mbox{ if } n=3j^2+4j+1 \mbox{ for some } j \geq 0,\\ 0, & \mbox{ otherwise.}\end{cases} $$
\end{theorem}

\begin{proof}

We write a  partition $\pi\in A_{\nu}(n)$  as $\pi=(\pi^e,\pi^o)$, where $\pi^e$ (resp.~$\pi^o$) is the partition consisting of the even (resp.~odd) parts of $\pi$. As usual, the largest part of $\pi^e$  is $\pi_1^e$. We denote by $\pi^o_s$ the smallest part of $\pi^o$ and by $m^o(\pi)$ the multiplicity of $\pi^e_1+1$ in $\pi^o$. We have $m^o(\pi)\geq 0$. 

Let $\tilde A_{\nu}(n)=\{\pi\in A_{\nu}(n) \mid \pi^o=(\pi^e_1+1)^{m^o(\pi)}, \mbox{ with } m^o(\pi)\in\{\frac{\pi^e_1}{2}, \frac{\pi^e_1}{2}+1\}\}$. Then, $|\tilde A_{\nu}(n)|=0$ or $1$.

We define an involution  on $A_{\nu}(n)\setminus \tilde A_{\nu}(n)$ as follows. 

(i) If $\pi^o_s \geq 2m^o(\pi)+1$, remove $\pi^e_1$ from $\pi^e$ and the last two columns (of length $m^o(\pi)$) from $\pi^o$, and add parts $\pi^e_1-1$ and $2m^o(\pi) +1$ to $\pi^o$. 

(ii) If $\pi^o_s < 2m^o(\pi)+1$, remove from $\pi^o$ one part equal to $\pi^e_1+1$ (largest part) and one part equal to $\pi^o_s$, and add a part equal to $\pi^e_1+2$ to $\pi^e$ and two columns equal to $\frac{\pi^o_s-1}{2}$. 

Note that the transformations in (i) and (ii) are inverses of each other. 

We have  $|\tilde A_{\nu}(n)|=1$ if and only if $n=3j^2+2j$ or $3j^2+4j+1$. Moreover, for $\pi\in \tilde A_{\nu}(n)$, $j=\ell(\pi^e)=\ell_e(\pi)$  which completely determines the parity of  $\ell(\pi)$. 

\end{proof}
\begin{corollary}\label{cor_pent} Let $n\in \mathbb N$. Then  $c(n)$ is odd if and only if $n$ is eight times a generalized pentagonal number.  \end{corollary}

\begin{proof}  We have $$c(n)=\sum_{\pi\in A_{\nu}(n)}(\ell_o(\pi)-1).$$ 
With the notation in the proof of Theorem \ref{pent},  we have  $\ell_o(\pi)=\ell(\pi^{o})\equiv n \pmod 2$ because the number of parts in a partition with odd parts has the same parity as its size. Therefore, if $n$ is odd, $\ell_o(\pi)-1$ is even for every $\pi\in A_{\nu}(n)$ and $c(n)$ is even. 

If $n$ is even, 
 $c(n) \equiv |A_\nu(n)|\pmod 2$. From Theorem \ref{pent}, it follows that $|A_\nu(n)|\equiv 1\pmod 2$ if and only if $n=3j^2+2j$ or $3j^2+4j+1$ for some $j \geq 0$. Since $n$ is even, if $n=3j^2+2j$, then $j$ must be even, and if  $n=3j^2+4j+1$, then $j$  must be odd. Therefore, $|A_\nu(n)|\equiv 1\pmod 2$ if and only if $n$ is eight times a generalized pentagonal number.
  \end{proof}
 \begin{remark} Theorem 4.3 also follows from setting $a=1$ in Entry 3.7 of \cite{BerndtYee}. Using Lemma \ref{lem_AnuandAnu2}, Theorem \ref{pent} can be adapted to a pentagonal number theorem for $A_{\nu,2}(n)$. %For each partition in  $A_\nu(n)$ the parity of the number of parts is determined by the parity of $n$ and the parity of the number of even parts. 
Then,  \cite[Section 3.2]{LY19} leads to a combinatorial proof of \cite[Theorem 5.4]{ADY}.
\end{remark}

\section{The mock theta function $\phi$}\label{sec_phi}
Recall from Section \ref{sec_mockintro} that the third order mock theta function $\phi$ is defined by
 \begin{align}\label{phi_defn} \phi(q):=\sum_{n=0}^\infty \frac{q^{n^2}}{(-q^2,q^2)_n} = \sum_{n=0}^\infty (sc_e(n) - sc_o(n)) q^n,\end{align} where $sc_{e}(n)$ (resp. $sc_{o}(n)$) counts the number of self-conjugate partitions $\lambda$ of $n$ with $L(\lambda)$ even (resp. odd).  Here, $L(\lambda)$ is the number of parts of $\lambda$ minus the side length of its Durfee square.  
{From \cite[Proof of Theorem 4.2]{ADY}, we have} $$\phi(q)=1+\sum_{n=0}^\infty (-1)^n q^{2n+1}(q;q^2)_n.$$   

We first define the following generalization of $\phi(q)$:
\begin{align*}
    B_{\phi}(z;q) &:= 1 + \sum_{n=0}^\infty z^n q^{2n+1}(q;q^2)_n  =\sum_{n=0}^\infty\sum_{m=0}^\infty b_{\phi}(m,n) z^m q^n,
\end{align*}  where $b_{\phi}(0,0):=1$, and for $(m,n)\neq (0,0)$,
$b_{\phi}(m,n)$ equals the difference between the number of partitions of $n$ into distinct odd parts with largest part $2m+1$ and an odd number of parts, and the number of such partitions with an even number of parts.
Note that $B_{\phi}(-1,q) = \phi(q),$   and that this gives rise to a different combinatorial interpretation for the coefficients of $\phi$ than the one given in \eqref{phi_defn}.  Namely, the coefficient of $q^n$ in the $q$-series expansion for $\phi(q)$ also equals $do_e(n)-do_o(n)$, where $do_e(n)$ (resp.~$do_o(n)$) counts the number of partitions of $n$ into distinct odd parts with $M_2$-rank even (resp.~odd). 

Next we define another bivariate function, which we later explain is  related to $\phi$ when $z=1$ (see
\eqref{phi_ABCD}):

\begin{align}\notag A_\phi(z;q) &:= q\sum_{n=0}^\infty z q^n (-z q^{n+1};q)_n (-z q^{2n+1};q^2)_\infty    =\sum_{n=0}^\infty\sum_{m=0}^\infty a_{\phi}(m,n) z^m q^{n+1},  \\
\end{align}   where $a_\phi(m,n)$ is the number of partitions of $n$ into distinct  parts, {with} $m$ parts, such that each even part is at most twice the smallest part. The function $A_\phi(z;q)$ 
is related to $\phi(q)$ by the following identity:
 \begin{align}\label{phi_ABCD}A_\phi(1;q) =  1-\phi(q) + 2(-q;q^2)_\infty \sum_{n=1}^\infty q^{n^2}.\end{align} 
 Using   Jacobi's triple product identity \cite[(2.2.10) with $z=1$]{AndrewsEncy},
identity \eqref{phi_ABCD}   is essentially \cite[Theorem 4.2]{ADY} with some minor typographical errors corrected. {Unlike the mock theta functions $\omega(q)$ and $\nu(-q)$ studied in Sections \ref{BeckIdentity} and \ref{sec_nusection}, the $q$-series coefficients of $\phi(q)$ are not uniformly non-negative, e.g.,
 $$\phi(q) = 1 + q - q^3 + q^4 + q^5 - q^6 - q^7 + 2 q^9 - 
 2 q^{11} + q^{12} + q^{13} - q^{14} - 2 q^{15} + q^{16} +   O(q^{17}).$$
However, the} authors of \cite{ADY} present \eqref{phi_ABCD}  for $\phi(q)$ as a companion identity to 
 their similar result \cite[Theorem 4.1]{ADY} (see \eqref{eqn_nuqBnu}) which  shows that the   mock theta function $\nu(-q)$   is equal to the generating function for partitions into 
  distinct parts, in which
 each odd part is less than twice the smallest part. Identity~\eqref{phi_ABCD} similarly relates  the mock theta function $\phi(q)$ to the generating function for partitions into distinct parts in which each even part is at most twice the smallest part, but up to a theta function. Indeed it is identity~\eqref{phi_ABCD} that leads to our ``Beck-type" Theorem \ref{thm_phimain} for the mock theta function $\phi(q)$ below.  To state it, we introduce the functions
  \begin{align} \label{def_F1q}
    F_1(q) &:= F_3(q)\Big(1+2 \sum_{n=1}^\infty q^{n^2}\Big) + 2(-q;q^2)_\infty \sum_{n=1}^\infty q^{n^2}, \\ \label{def_F2q} 
    F_2(q) &:= (-q;q^2)_\infty \sum_{m=1}^\infty \frac{q^{2m-1}}{1+q^{2m-1}}, \ \\
  F_3(q) &:= (-q;q^2)_\infty \sum_{m=1}^\infty \frac{q^{2m}}{1+q^{2m}}.
    \end{align} 
The functions $F_2(q)$ and $F_3(q)$, including their combinatorial interpretations, are studied in \cite{WIN5Lehmer}.

{In what follows, we use the notation $G(q) \succeq_{S} 0$, where $S\subseteq \mathbb N$, to mean that when expanded as a $q$-series, the coefficients of $G(q)$ are non-negative, with the exception of the coefficients of $q^{n}$ for $n\in S$. When $S=\emptyset,$ we simply use the notation $\succeq 0$.}

\begin{theorem}\label{thm_phimain}We have \begin{align}\label{eqn_ddAB1Phi} 
 \frac{\partial}{\partial z} \Big |_{z=1} (A_\phi(z;q) + B_{\phi}(-z^{-1};q)) = F_1(q) - F_2(q).\end{align}
    Moreover, we have 
     $$\frac{\partial}{\partial z} \Big |_{z=1} (A_\phi(z;q) + B_{\phi}(-z^{-1};q)) \succeq 0.$$
\end{theorem}
{{\begin{remark}  A combinatorial interpretation of   \eqref{eqn_ddAB1Phi} in Theorem \ref{thm_phimain} can be deduced from the combinatorial definitions of $a_\phi(m,n)$ and $b_{\phi}(m,n)$ provided above, together with combinatorial interpretations of {the} $q$-series coefficients  {in} the functions $F_2(q)$ and $F_3(q)$ provided in \cite{WIN5Lehmer}, and the definition of $F_1(q)$.   While this combinatorial interpretation involves partition differences, Theorem \ref{thm_phimain}  establishes the non-negativity of the $q$-series coefficients of \eqref{eqn_ddAB1Phi}.  On the other hand, it is of interest to find another proof of this fact by finding a different and \emph{manifestly} positive combinatorial interpretation of the $q$-series coefficients of $F_1(q)-F_2(q)$ (i.e., one which does not involve a combinatorial difference).  We leave this as an open problem. \end{remark}
}}

\subsection{Proof of Theorem \ref{thm_phimain}}
In this section, we prove Theorem \ref{thm_phimain}, assuming the truth of Proposition \ref{prop_as} and Proposition \ref{prop_bs} stated below.  We provide a combinatorial proof of Proposition \ref{prop_as} in  Section \ref{sec_proofpropas}, and provide both   combinatorial and analytic proofs of Proposition \ref{prop_bs} in Section \ref{sec_proofpropbs}. 
 
 \begin{proposition} \label{prop_as} We have  $$2 F_3(q) - F_2(q) \succeq_{S} 0,  $$ where $S:=\{1,4,8,16\}.$  Moreover, the coefficients of $2 F_3(q) - F_2(q)$ are at least 4, with the exception of the coefficients of $q^n$ with $n$ in the set $U:=\{1,2,3,4,5,8, 9, 12, 13,16,17\}$.  
 
    \end{proposition} 
    
\begin{corollary}\label{cor_p1}  We have 
    \begin{align}\label{def_sand2} (2F_3(q) - F_2(q))\sum_{n=1}^\infty q^{n^2} \succeq_T 0,\end{align} where $T:=\{2,5,9,13,17\}$.
\end{corollary}
\begin{proof}[Proof of Corollary \ref{cor_p1}]
    Writing $2F_3(q) - F_2(q) =: \sum_{n=1}^\infty a_n q^n,$ we have the coefficient of $q^n$ in $(2F_3(q) - F_2(q))\sum_{j=1}^\infty q^{j^2}$ is equal to
    \begin{align}\label{eqn_sumak} \mathop{\sum_{k=1}^{n-1} a_{k}}_{k+m^2=n, \ 1\leq m^2 \leq n-1}.\end{align}  
    By Proposition \ref{prop_as},  we have $a_k\geq 0$ for any $k\not\in S = \{1,4,8,16\}$, and $a_k\geq 4$ for any $k\not\in U$. 
    
Let $n\notin T$ be a positive integer. 
We can directly compute the $q$-series in \eqref{def_sand2} to $O(q^{65})$ to see the coefficients are non-negative. (Here and throughout we use standard ``big-$O$," also called ``big-oh," notation; see e.g.\ \cite{Apostol}.)
Now assume that $n>65$.
To prove that \eqref{eqn_sumak} is non-negative, it suffices to show that for any $k\in S$ with $k+m^2=n$, there is another $k'\neq k$ with $k'+m'^2=n$ such that $a_{k'} + a_k \geq 0$. 
Because $n>65$, we must have $m\geq2$.
Let $k' :=k+2m-1$. Then $1\leq k' \leq n-1$, $k'>k$, 
     and $k'+(m-1)^2=n$, where $1\leq (m-1)^2 \leq n-1$.    
Note that $k' \not\in U$.
This is because for $k\in S$ and $k'\in U$, we have $n=k+m^2 = k+(\frac{k'-k+1}2)^2\leq 65$. 
By direct calculation, we find that $\min\{a_k\}_{k\in S} = -4,$  and hence  $a_k + a_{k'} \geq 0$. 
\end{proof}

\begin{proposition}\label{prop_bs} {Let $F_2(q) =: \sum_{n=1}^\infty b_n q^n$}.  For $n\geq 9$, we have $$b_n \leq b_{n-1} + b_{n-4}.$$  
\end{proposition}

\begin{corollary}\label{cor_p2}  We have 
 \begin{align}\label{def_sand3} F_2(q) \sum_{n=1}^\infty q^{n^2} - F_2(q)  \succeq_V 0,
    \end{align}  where $V:=\{1,3,4,6,8\}$. 
\end{corollary}
\begin{proof}[Proof of Corollary \ref{cor_p2}]  To prove \eqref{def_sand3}, we show \begin{align}\label{eqn_sumbk} \sum_{1\leq m^2\leq n-1}  b_{n-m^2} \geq b_n\end{align} for $n \in \mathbb N \setminus V$. By Proposition \ref{prop_bs}, and the non-negativity of $b_n$, for $n\geq 9$, we have $$\sum_{1\leq m^2\leq n-1}  b_{n-m^2} \geq  b_{n-1}+b_{n-4} \geq b_n.$$ For $n\in\{2,5,7\}$, the inequality \eqref{eqn_sumbk}  can be verified directly.
\end{proof}

\subsubsection{Proof of Theorem \ref{thm_phimain}}\label{sec_proofphimain}
First, by straightforward manipulations, we find that \cite[(2.4)]{LY19} leads to the identity    
 \begin{align}\label{eqn_bphiid} B_\phi(-z^{-1};q) = 1 + \sum_{n=0}^\infty \frac{z^{-n} q^{(n+1)^2}}{(-z^{-1}q^2;q^2)_{n+1}}.\end{align} {We recall \cite[Theorem 6.11]{BDG}, which gives the interesting identity
\begin{align}\label{eqn_bdg611}  
\sum_{n=0}^\infty &q^n (-z q^n;q)_{n+1}(-z q^{2n+2};q^2)_\infty \\ \notag
&=-\frac{q}{z}\nu\left(\frac{q^2}{z},-\frac{q^2}{z};-q\right) + \frac{1}{q}(-z;q^2)_\infty\left(-1+\frac{(-q;q)_\infty (q^2;q^2)_\infty (-q^2/z;q^2)_\infty}{(-q^3/z;q^2)_\infty}\right),
\end{align}
where the function $\nu(\alpha,z;q)$ is defined in \cite[(1.18)]{BDG} by
$$\nu(\alpha,z;q) := \sum_{n=0}^\infty \frac{\alpha^n q^{n^2+n}}{(-z q;q^2)_{n+1}}.$$
We let  $z\mapsto zq$  in \eqref{eqn_bdg611} and multiply the resulting identity by $zq$.  Using this and  \eqref{eqn_bphiid},} we find that
$${A}_{\phi}(z;q) + {B}_{\phi}(-z^{-1};q)
= D_{\phi}(z;q),$$
where
\begin{align*}D_{\phi}(z;q) &:= 1+z(-zq;q^2)_\infty\left(-1 + \frac{(-q;q)_\infty (q^2;q^2)_\infty (-q/z;q^2)_\infty}{(-q^2/z;q^2)_\infty}\right) \\
&= 1-z(-zq;q^2)_\infty + z\frac{(-q;q)_\infty}{(-q^2/z;q^2)_\infty}\left(1+\sum_{n=1}^\infty(z^n + z^{-n})q^{n^2}\right).
\end{align*}  Above, we have also used  the Jacobi triple product  \cite[(2.2.10)]{AndrewsEncy}.
Thus, the derivative difference on the left hand side of  \eqref{eqn_ddAB1Phi} equals
$
 \frac{\partial}{\partial z} \big |_{z=1} {D}_{\phi}(z;q).
$  After a direct calculation using the definition of $D_{\phi}(z;q)$ and some simplification, we obtain that this derivative difference equals $F_1(q) - F_2(q)$.  

To prove the second assertion of the theorem, it now suffices to show that $F_1(q) - F_2(q) \succeq 0.$  
 From \cite{WIN5Lehmer}, we have  $$F_3(q) \succeq 0.$$  From this and \eqref{def_F1q}, it is not difficult to see that 
\begin{align}\label{def_sand1} F_1(q) - 2 F_3(q) \sum_{n=1}^\infty q^{n^2} \succeq 0.\end{align}  
    Thus, we have from  \eqref{def_sand2},   \eqref{def_sand3}, and \eqref{def_sand1} that  $F_1(q) - F_2(q) \succeq_W 0$ for some explicit, finite, set $W$.     The proof is complete after a direct calculation  of the $q$-series for $F_1(q)-F_2(q)$ up  to $O(q^{n_W})$, where $n_W:= \max W $, which reveals that, in fact, $F_1(q) - F_2(q) \succeq 0$  as claimed.

\subsection{Proof of Proposition \ref{prop_as}}\label{sec_proofpropas}

     Setting $r=1$ and $\ell=1$ in \cite[Section 5.2]{WIN5Lehmer} shows that  $F_2(q)$ is the generating function for $|A(n)|$, where $$A(n):=\{(\lambda, (a))\vdash n \ \big| \ \lambda \in \mathcal Q_o, a \mbox{ odd}, a\not \in \lambda\}.$$
    Similarly, setting $r=1$, $\ell=2$ in \cite[Section 5.2]{WIN5Lehmer} shows that $F_3(q)$  is the generating function for $|B(n)|-\varepsilon(n)$, where  $$\varepsilon(n):=\begin{cases}1, & \text{if } n\equiv 0\pmod 4,\\ 0, & \text{else,} \end{cases}$$   and  $$B(n):= \{(\lambda, (c^d))\vdash n \ \big| \ \lambda \in \mathcal Q_o,
	c \text{ even, } d \text{ odd, }
	 \lambda_1-\lambda_2\leq c, \text{ and }
	  \lambda \neq \mu(c) { \text{ if } c\geq4} \}.$$ 
   Here,  $\mu(c)$ is   defined for
	   even $c\geq 4$ to be the partition 
	  $$\mu(c)=\begin{cases}(\frac{c}{2}+1, \frac{c}{2}-1), & \mbox{ if } c\equiv 0\pmod 4,\\ (\frac{c}{2}+2, \frac{c}{2}-2), & \mbox{ if } c\equiv 2\pmod 4.\end{cases}$$ Thus, $\mu(c)$ is a partition of $c$ into two distinct odd parts with smallest possible difference between the parts. {We remark that, if $n\not \equiv 0\pmod 4$, $c$ even, and $d$ odd, then $(\lambda, (c^d))\neq (\mu(c), (c^d))$ vacuously.  }

{We will show that  $2(|B(n)|-\varepsilon(n))-|A(n)|\geq 4$ for  all $n\not \in S$.}
     For $n<53$, this can be verified directly.    For the remainder of the proof, let  $n\geq 53$. \smallskip 
  {\ \\ {
   {\underline{Roadmap of the proof}.}  Since the proof is intricate, we begin by providing a roadmap for the benefit of the reader.  Ultimately, we   show that  $2|B(n)|\geq|A(n)|+6$, which is sufficient to prove the proposition.    To do this, we   establish   relevant injections.  To describe them, we denote by $B'(n)$  the multiset whose elements are precisely those of $B(n)$, each appearing with multiplicity $2$. Equivalently,  $B'(n)$ is the disjoint union of two copies of $B(n)$. We also 
  let $$U(n):=\{(\lambda, (c^d))\vdash n \mid \lambda \text{ has odd parts, }
	c \text{ even, } d \text{ odd}, \lambda_1-\lambda_2\leq c\}.$$  Note that  $B(n)\subseteq U(n)$.  Finally, we 
 let $U'(n)$ be the multiset whose elements are precisely those of $U(n)$, each appearing with multiplicity $2$. 
In the proof of the proposition below, we define an injection $\psi$ from $A(n)$ to $U'(n)$ as a composition of two mappings $\Psi_1$ (Step 1) and $\Psi_2$ (Step 2). {At the end of Step 2, we describe the image of $\psi$. } Then  we  define an injection $\zeta$ (Step 3) $$ \psi(A(n))\setminus B'(n) \to B'(n)\setminus \psi(A(n))$$   and show that $$|B'(n)\setminus \psi(A(n))|\geq |\psi(A(n))\setminus B'(n)|+6$$ by explicitly listing six elements in $B'(n)\setminus \psi(A(n))$ that are not in the image of $\zeta$.
 
 %In the proof of the proposition below, we define an injection $\psi$ from $A(n)$ to $U'(n)$ as a composition of two mappings $\Psi_1$ and $\Psi_2$. Then we  define an injection $$\zeta: \psi(A(n))\setminus B'(n) \to B'(n)\setminus \psi(A(n))$$ and show that $$|B'(n)\setminus \psi(A(n))|\geq |\psi(A(n))\setminus B'(n)|+6$$ by explicitly listing six elements in $B'(n)\setminus \psi(A(n))$ that are not in the image of $\zeta$. 
 This shows that $|B'(n)|\geq |\psi(A(n))| +6$, which is equivalent to $2|B(n)|\geq |A(n)|+6,$ and implies $2(|B(n)|-\varepsilon(n))\geq |A(n)|+4$ as desired.  %The proofs of these injections follow below, and are broken up into two main steps, Step 1 and Step 2.  Step 2 is further broken down into several cases.
 The details of the proof of the proposition begin now.}}

Let \begin{align*}A_1(n) & =\{(\lambda, (a))\in A(n) \mid 1\in \lambda\},\\
A_2(n) & =\{(\lambda, (a))\in A(n) \mid 1\not\in \lambda, a\neq 1\},\\
A_3(n) & =\{(\lambda, (a))\in A(n) \mid 1\not\in \lambda, a= 1\}.
\end{align*} Then, $A(n) = A_1(n)\sqcup A_2(n) \sqcup A_3(n)$. We also define 
\begin{align*}C_1(n) & =\{(\eta, (c))\vdash n \mid c \text{ even, } c\geq 4, \eta\in \mathcal Q_o, 1\not \in \eta, c-1\not \in \eta\},\\ C_2(n) & =\{(\eta, (c))\vdash n \mid c \text{ even, } c\geq 2,   \eta\in \mathcal Q_o, 1 \in \eta, c+1\not \in \eta\},
\end{align*} and set $C(n)=C_1(n)\sqcup C_2(n)$. 
	
	We define $\psi:A(n)\to U'(n)$ as $\psi=\Psi_2\circ \Psi_1$, where    $\Psi_1:A(n)\to C(n)$ and $\Psi_2: C(n)\to U'(n)$ are defined in the steps below. 	
	\smallskip \\ \noindent {\bf Step 1:}  $\Psi_1:A(n)\to C(n)$. Given $(\lambda, (a))=((\lambda_1, \lambda_2, \ldots, \lambda_{\ell(\lambda)}), (a)) \in A(n)$, we define  $$\Psi_1(\lambda, (a))=(\eta, (c))=\begin{cases} ((\lambda\setminus (1), (a+1)), & \mbox{ if }  (\lambda, (a))\in A_1(n), \\ (\lambda\cup (1), (a-1)),& \mbox{ if }  (\lambda, (a))\in A_2(n), \\ (\lambda\setminus (\lambda_1), (\lambda_1+1)),& \mbox{ if } (\lambda, (a))\in A_3(n).
	 \end{cases}$$

	 Then $\Psi_1$ induces bijections $\Psi_1:A_1(n)\to C_1(n)$ and $\Psi_1:A_2(n)\to C_2(n)$,  and the injection $\Psi_1:A_3(n)\to C_1(n)$  whose image is $$C_1'(n)=\{(\eta, (c))\in C_1(n)\mid \eta_1\leq c-3\}.$$ We note that for $n\geq 4$ even, the pair $(\emptyset, (n))$ belongs to $C'_1(n)$ and thus also to $C_1(n)$. 
\smallskip \\	\noindent {\bf Step 2:}    $\Psi_2: C(n)\to U'(n)$. 
	Let $(\eta, (c))\in C(n)$. 
% \cycom{With the convention of $q_\eta=r_\eta=0$ for $\eta=\emptyset$ below, we don't have to define $\Psi_2(\eta, (c))$ separately for $\eta=\emptyset$. I marked in red the words that can be omitted.}
 %{\color{red}If $\eta=\emptyset$,  we define $\Psi_2(\emptyset, (n))=(\emptyset, (n))$. }
 If $\eta\neq\emptyset$, write   $\eta_1-\eta_2=q_\eta c+r_\eta$ with $q_\eta, r_\eta \in \mathbb Z$, $q_\eta \geq 0$, $0<r_\eta\leq c$. We use  the convention that $\eta_j=0$ for all $j>\ell(\eta)$. If $\eta=\emptyset$, let $q_\eta = r_\eta = 0$.

	For $i=1,2$, we write  $C_i= C_{i,e}\sqcup C_{i,o}$, where \begin{align*}C_{i,e}(n)& =\{(\eta, (c))\in C_i(n)\mid q_\eta \text{ even}\},\\ C_{i,o}(n)& =\{(\eta, (c))\in C_i(n)\mid q_\eta \text{ odd}\}. \end{align*} 
	
	Moreover,  if $(\eta, c)\in C'_1(n)$, then $q_\eta=0$ and thus $C'_1(n)\subseteq C_{1,e}(n)$.  %We write $C'_{1,e}(n)= \cap C'_1(n)$ to emphasize that $q_\eta$ is even for pairs in this set,  

	 Set  $\eta-(q_\eta c):=(\eta_1-q_\eta c, \eta_2, \ldots, \eta_{\ell(\eta)})$. Note that $\eta-(q_\eta c)\in \mathcal Q_0$. 
  For simplicity, we write $q$ for $q_\eta$. 
	 
	 For $(\eta, (c))\in C(n)$  we define 
	 $$\Psi_2(\eta,(c))=(\xi, (c^d))=\begin{cases} (\eta-(q c), (c^{q+1})), & \text{ if } (\eta,(c)) \in C_{1,e}(n)\cup C_{2,e}(n),\\ ((\eta-(q c))\cup (c-1)\cup(1), (c^{q})), & \text{ if } (\eta,(c)) \in  C_{1,o}(n),\\ ((\eta-(q c))\setminus(1)\cup (c+1), (c^{q})),  & \text{ if } (\eta,(c)) \in  C_{2,o}(n).
 \end{cases}$$

All pairs %$(\xi, (c^d))$ 
obtained  satisfy $\xi_1-\xi_2\leq c$.

We determine the image of the relevant subsets of $C(n)$ under $\Psi_2$. Notice that $\Psi_2$ is an injection on each of the relevant subsets. We have

$$\Psi_2(C_{1,e}(n))=\left\{(\xi, (c^d))\in U(n) \left| \begin{array}{l} \xi \in \mathcal Q_o, c\geq 4, \text{ and }
\\ ((1\not\in \xi,   c-1\not \in \xi) \text{ or }\\ \quad\quad\quad\quad (1\not\in\xi, \xi_1=c-1, d>1) \text{ or } (\xi=(1),d>1))\end{array}\right.\right\},  $$

$$\Psi_2(C_{2,e}(n))=\left\{(\xi, (c^d))\in U(n) \Big| \begin{array}{l}  \xi \in \mathcal Q_o, 1\in \xi, \text{ and } (  c+1\not \in \xi \text{ or } ( \xi_1=c+1, d>1)), \\ \hspace{2.5cm} \text{ and if } \ell(\xi)= 1 \text{ then } d=1\end{array}\right\},$$

 $$\Psi_2(C_{1,o}(n))=\left\{(\xi, (c^d))\in U(n) \Bigg| \begin{array}{l} 1\in \xi, c\geq 4, \,   c-1\in \xi \text{ and } \ell(\xi)\geq 3 \text{ and } \\  \xi \not \in \mathcal Q_o \implies (\xi_1=\xi_2=c-1 \text{ and } \xi\setminus (c-1)\in \mathcal Q_o) \text{ or } \\ \hspace{2.5cm} \xi=(c-1,1,1)  \end{array} \right\},$$
 
 $$\Psi_2(C_{2,o}(n))=\left\{(\xi, (c^d)\in U(n) \Bigg| \begin{array}{l} 1\not\in \xi \text{ and } c+1\in \xi \text{ and } \ell(\xi)\geq 2 \text{ and } \\ \xi \not \in \mathcal Q_o \implies \xi_1=\xi_2=c+1 \text{ and } \xi\setminus (c+1)\in \mathcal Q_o
 \end{array} \right\}.$$
 Moreover, $\Psi_2$ is the identity on $C'_{1}(n)$, so
 $$\Psi_2(C'_{1}(n))=C'_{1}(n)=\{(\xi, (c))\in U(n) \mid \xi\in \mathcal Q_o, c\geq 4, 1\not \in \xi, \xi_1\leq c-3\} .$$
Thus, $\psi(A(n))$ is the multiset $$\psi(A(n))=\Psi_2(C_{1,e}(n))\sqcup \Psi_2(C_{2,e}(n)) \sqcup \Psi_2(C_{1,o}(n))\sqcup \Psi_2(C_{2,o}(n))\sqcup \Psi_2(C'_{1}(n)).$$

To find  the pairs $(\xi, (c^d))\in \psi(A(n))$ occurring   with multiplicity $2$, we determine the mutual intersections of the images under $\Psi_2$ of the different subsets of $C(n)$. 

%$\Psi_2(C_{1,e}(n))\cap \Psi_2(C_{2,e}(n))=\emptyset$,\quad  $\Psi_2(C_{1,e}(n))\cap \Psi_2(C_{1,o}(n))=\emptyset$, \quad $\Psi_2(C_{2,e}(n))\cap \Psi_2(C_{2,o}(n))=\emptyset$,\quad $\Psi_2(C_{2,e}(n))\cap \Psi_2(C'_{1,e}(n))=\emptyset$, \quad $\Psi_2(C_{1,o}(n))\cap \Psi_2(C_{2,o}(n))=\emptyset$, \quad $\Psi_2(C_{1,o}(n))\cap \Psi_2(C'_{1,e}(n))=\emptyset$,  \quad $\Psi_2(C_{2,o}(n))\cap \Psi_2(C'_{1,e}(n))=\emptyset$. 
%\ccom{Not sure we need to list all empty intersections. We could just list the intersections below and write that all other intersections are empty. }

We have 
\begin{align*}  \Psi_2(C_{1,e}(n))\cap &  \Psi_2(C_{2,o}(n))=\\ & \{(\xi, (c^d))\in U(n) \mid \xi\in\mathcal Q_o, 1\not\in \xi, c\geq 4,\ c+1\in \xi,  c-1\not \in \xi \text{ and } \ell(\xi)\geq 2\}, \\  \ \\ 
\Psi_2(C_{1,e}(n))\cap &   \Psi_2(C'_{1}(n)) \\ & =  \Psi_2(C'_{1}(n))=\{(\xi, (c))\in U(n) \mid \xi\in \mathcal Q_o, c\geq 4, 1\not \in \xi, \xi_1\leq c-3\},   \\ \ \\ 
\Psi_2(C_{2,e}(n))\cap &  \Psi_2(C_{1,o}(n))=\\ & \left\{(\xi, (c^d))\in U(n) \Big| \begin{array}{l} \xi\in \mathcal Q_o, 1\in \xi, c\geq 4, c-1\in \xi, \ell(\xi)\geq 3, \text{ and }\\ \quad \quad   ((c+1\not \in \xi) \text{ or } (\xi_1=c+1, d>1) )\end{array}\right\}.   \end{align*} 
 One can easily verify that all other mutual intersections are empty. Therefore, the mapping $\psi$ is a multiset  injection. 
Let $\mathcal S(n)$ be the union of the three sets above, i.e., 
{$\mathcal S(n)$ is the (disjoint) set of pairs $(\xi, (c^d))\in U(n)$ with $\xi \in \mathcal Q_o$, $c\geq 4$ and 
\begin{itemize}
\item if $1\not \in \xi$, then $(c+1\in \xi,  c-1\not \in \xi \text{ and } \ell(\xi)\geq 2)$ or $(d=1 \text{ and } \xi_1\leq c-3)$;
\item if $1 \in \xi$, then $c-1\in \xi, \ell(\xi)\geq 3, \text{ and } ((c+1\not \in \xi) \text{ or } (\xi_1=c+1, d>1)).$
\end{itemize}}
%\begin{align*}\mathcal S(n)= & \{(\xi, (c^d))\in U(n) \mid \xi \in \mathcal Q_o, 1\not\in \xi, c\geq 4,\ c+1\in \xi,  c-1\not \in \xi \text{ and } \ell(\xi)\geq 2\}\\ &  \cup \{(\xi, (c))\in U(n) \mid \xi\in \mathcal Q_o, c\geq 4, 1\not \in \xi, \xi_1\leq c-3\}\\&  \cup \left\{(\xi, (c^d))\in U(n) \Big| \begin{array}{l} \xi\in \mathcal Q_o, 1\in \xi, c\geq 4, c-1\in \xi, \ell(\xi)\geq 3), \text{ and }\\ \quad \quad   ((c+1\not \in \xi) \text{ or } (\xi_1=c+1, d>1))\end{array}\right\}\end{align*}
Thus, the elements of $\psi(A(n))$ occurring twice are precisely the elements of $\mathcal S(n)$. We denote by $\mathcal S'(n)$ be the multiset whose elements are precisely those of $\mathcal S(n)$, each appearing with multiplicity $2$. Then, every element in $\psi(A(n))\setminus \mathcal S'(n)$ has multiplicity $1$. 
\smallskip \\	\noindent {\bf Step 3:}  Recall that our goal is to prove that $|B'(n)|\geq |\psi(A(n))| +6$. To this end, we now show that $|B'(n)\setminus \psi(A(n))|\geq |\psi(A(n))\setminus B'(n)|+6$.

 As a set, $\psi(A(n))\setminus B'(n)$ consists of \textit{precisely} the pairs $(\xi, (c^d))\in U'(n)$ satisfying one of the following conditions: 

\begin{itemize}

\item[(i)] $\xi=\mu(c)$; 

\item[(ii)] $\xi_1=\xi_2=c-1$, $\xi\setminus (c-1)\in \mathcal Q_o$ and $1 \in \xi$ and $c\geq 4$;  %\ccom{Note that $\xi\setminus\{\xi_1\}\in\mathcal Q_o$.}

   \item[(iii)] $\xi=(c-1,1,1)$, $c\geq 4$;

 \item[(iv)]$\xi_1=\xi_2=c+1$, $\xi\setminus (c+1)\in \mathcal Q_o$ and $1 \not \in \xi$; 
   
%   \item[(iv)] $\xi=(1,1)$, $d>1$. \ccom{remove??}

\end{itemize}

Since $B(n)=\{(\lambda, (c^d))\in U(n)\mid \lambda \text{ has distinct parts}, \lambda\neq \mu(c)\}$, it is clear  these pairs do not belong to $B'(n)$ and  all the pairs in $\psi(A(n))$ that are not in $B'(n)$ must satisfy one of (i)-(iv).

We define a multiset injection $\zeta: \psi(A(n))\setminus B'(n)\to B'(n) \setminus \psi(A(n))$. 
In the process, we also clarify  that the pairs satisfying (i)-(iv) above belong to $\psi(A(n))$ and discuss their multiplicities in $\psi(A(n))\setminus B'(n)$. We note that %$B'(n) \setminus \psi(A(n))$ consists of one copy of each pair in $B(n)\cap(\psi(A(n))\setminus \mathcal S'(n))$  and two copies of each pair in $B(n)\setminus \psi(A(n))$. \bcom{Since a lot of the arguments below for the image being appropriate are like ``the image is not in $\mathcal{S}(n)$'', I wonder if we could say something like 
 $B'(n) \setminus \psi(A(n))$ consists of one copy of each pair in $B(n)\setminus \mathcal{S}(n)$ and an additional copy of each pair in $B(n)\setminus \psi(A(n))$.
By inspection, we see that the  pairs $(\xi, (c^d))\in B(n)$ satisfying  one of the conditions below  are \textit{not} in $\psi(A(n))$:
 $$\def\arraystretch{1.3}\begin{array}{l} 
\bullet \  1\in \xi, \ c+1  \in \xi, \ c-1\not \in \xi, \ \xi_1>c+1;  \\  \bullet \  c=2, \ 1\in \xi, \ 3  \in \xi, \ \xi_1>3;  \\ \bullet \ 
  1\not \in \xi, {c \geq 4,}\ c-1  \in \xi, \ c+1\not \in \xi, \ \xi_1>c+1.
 \end{array}$$ 
The list above is not exhaustive but it is sufficient for our purposes. We denote by $\mathcal T(n)$ the set containing  each pair $(\xi, (c^d))\vdash n$ satisfying the conditions above and by $\mathcal T'(n)$ the multiset whose elements are precisely those of $\mathcal T(n)$, each appearing with multiplicity $2$.

Let $(\xi, (c^d))\in \psi(A(n))\setminus B'(n)$. {We note that when pairs are mapped by $\zeta$ to $\mathcal T'(n)$ the assignment is \textit{ad hoc}. Moreover, the pairs in $\mathcal T'(n)$ that are in the image of $\zeta$ are of the form $(\mu(n-k)\cup \alpha, (x^y))$, where $\alpha$ is a partition with small parts and $|\alpha|+xy=k$. Since $n\geq 53$, in all such pairs the parts of $\mu(n-k)$ are  larger than the parts of $\alpha$. This is not a necessary condition but it allows for an easy check that the map $\zeta$ is indeed an injection. In what follows, the congruence conditions on $n$ are imposed by the requirement that, if $(\xi, (c^d))\in \psi(A(n))\setminus B'(n)$, then   $c$ is even and $d$ is odd.   }\medskip

\underline{Case (i)} $\xi=\mu(c)$. Since $c$ is even and $d$ is odd, if $(\mu(c), (c^d))\vdash n$, then $n\equiv 0 \pmod 4$. \smallskip

If $d=1$, since $n\geq 53$, it follows that $c\geq 28$. Then   $(\mu(c), (c))\in \Psi_2(C_{1,e}(n))\cap    \Psi_2(C'_{1}(n))$. 
 We define $$\zeta(\mu(n/2),(n/2)):= (\mu(n-10)\cup(5,1),(4))\in \mathcal T'(n).$$ Notice that we are mapping each copy of $(\mu(n/2),(n/2))$ to one of two copies of $(\mu(n-10)\cup(5,1),(4))$. As we will see below, this is the only pair that occurs twice in $\psi(A(n))\setminus B'(n)$. 

\medskip  

If $d>1$, $c=4$, then  $(\mu(4), (4^d))=((3,1), (4^{(n-4)/4}))\in \Psi_2(C_{2,e}(n)){\setminus\mathcal{S}(n)}$. Thus, $n\equiv 0 \pmod 8$, and we define $$\zeta((3,1), (4^{(n-4)/4})):=(\mu(n-14)\cup(7,1),(6))\in \mathcal T'(n).$$

If $d>1$, $c=6$, then $(\mu(6), (6^d))=((5,1), (6^{(n-6)/6})\in \Psi_2(C_{2,e}(n)){\setminus\mathcal{S}(n)}$. Thus,  $n\equiv 0 \pmod{12}$, and we define $$\zeta((5,1), (6^{(n-6)/6})):=(\mu(n-14)\cup(7,1),(6))\in \mathcal T'(n).$$

If $d>1$, $c\geq 8$, then $(\mu(c), (c^d))\in \Psi_2(C_{1,e}(n)){\setminus\mathcal{S}(n)}$.  We define $$\zeta(\mu(c), (c^d)):=(\mu(3c), (c^{d-2})).$$
The parts of $\mu(3c)$ are larger than $c+1$ and thus $(\mu(3c), (c^{d-2}))\not \in \mathcal{S}(n)$. Therefore,  $(\mu(3c), (c^{d-2}))\in B'(n)\setminus\psi(A(n))$. Moreover, $(\mu(3c), (c^{d-2}))\not\in \mathcal T(n)$. 

 \bigskip

\underline{Case (ii)}  If $\xi_1=\xi_2=c-1$, $\xi\setminus (c-1)\in \mathcal Q_o$,  $1 \in \xi$, and $c\geq 4$, then $(\xi, (c^d))\in \Psi_2(C_{1,o}(n)){\setminus\mathcal{S}(n)}$. 
\smallskip

To define $\zeta$, we consider two subcases.   \medskip

(I) $\xi_3<c-3$. We  define
$$\zeta(\xi, (c^d)):=(\xi\setminus (c-1,c-1)\cup \mu(2c-2), (c^d))=:(\nu, (c^d)).$$

Since $\ell(\xi)\geq 3$, $1\in \xi$,  and $\mu(2c-2)=(c+1, c-3)$, we must have $c>4$. Since $1\in \nu$ and $c-1 \not \in \nu$, it follows that $(\nu, (c^d))\not \in \mathcal S(n)$. Since $c+1\in\nu$ is the largest part,  it follows that $(\nu, (c^d))\not \in \mathcal T(n)$.\medskip

(II)  $\xi_3=c-3$. 
If $c=4,6,8,10$, we define $\zeta(\xi, (c^d))\in \mathcal T'(n)$ 
in an \textit{ad hoc} manner by the table below:

$$\begin{array}{l|l|l} n& (\xi, (c^d))&\zeta(\xi, (c^d))\smallskip  \\ \hline 3\pmod 8 & ((3,3,1),(4^{(n-7)/4}))& (\mu(n-11)\cup(5,3,1), (2))\smallskip \\ \hline 8\pmod {12} & ((5,5,3,1),(6^{(n-14)/6})) & (\mu(n-18)\cup(9,1),(8)) \smallskip \\ \hline 12\pmod {16} & ((7,7,5,1),(8^{(n-20)/8})) & (\mu(n-26)\cup(13,1),(12)) \smallskip \\  \hline 15\pmod{16} & ((7,7,5,3,1),(8^{(n-23)/8}))& (\mu(n-11)\cup(5,3,1), (2))\smallskip \\ \hline 16\pmod{20} & ((9,9,7,1),(10^{(n-26)/10})) & (\mu(n-26)\cup(13,1),(12)) \smallskip \\ \hline 4\pmod{20} & ((9,9,7,5,3,1),(10^{(n-34)/10})) & (\mu(n-26)\cup(13,1),(12)) \smallskip \\   \hline 1\pmod{20}& ((9,9,7,5,1),(10^{(n-31)/10}))& (\mu(n-29)\cup(7,3,1), (2^9))\smallskip \\
\hline 19\pmod{20}& ((9,9,7,3,1),(10^{(n-29)/10}))& (\mu(n-27)\cup(5,3,1), (2^9))\smallskip \end{array}
$$ {Note that the congruence conditions on $n$ imply that no pair in the right hand column occurs more than twice. }

For the remaining pairs, i.e., $c\geq 12$, we define

$$\zeta(\xi, (c^d)):=(\xi\setminus(\xi_1, \xi_2,\xi_3, 1) \cup \mu(3c-4) ,(c^d))=:(\nu, (c^d)).$$

Since $c\geq 12$,  $\mu(3c-4)$ has parts greater than $c+1$. Thus $1, c+1\not \in \nu$, and therefore $(\nu, (c^d))\not \in \mathcal S(n)$. Since $1, c-1\not \in \nu$, we have $(\nu, (c^d))\not \in \mathcal T(n)$.

\bigskip

\underline{Case (iii)} If $\xi=(c-1,1,1)$, $c\geq 4$, then $(\xi, (c^d))\in \Psi_2(C_{1,o}(n)){\setminus\mathcal{S}(n)}$.\medskip 

If $d=1$, we define $$\zeta(((n-3)/2, 1,1) , ((n-1)/2)):=(\mu(n-13)\cup(7,3,1), (2))\in \mathcal T'(n).$$

If $d=3$, we define $$\zeta(((n-5)/4, 1,1) , ((n-1)/4)^3):=(\mu(n-13)\cup(7,3,1), (2))\in \mathcal T'(n).$$

If $d>3$,  we define $$\zeta((c-1, 1,1), (c^d))=(\mu(5c)\cup (1), (c^{d-4})).$$ 
Since $c\geq 4$, the parts of $\mu(5c)$ are larger than $c+1$ and thus $(\mu(5c)\cup(1), (c^{d-2}))$ is neither in $S(n)$ nor in $\mathcal T(n)$. \bigskip

 \underline{Case (iv)} If $\xi_1=\xi_2=c+1$, $\xi\setminus (c+1)\in \mathcal Q_o$, and $1 \not \in \xi$, then $(\xi, (c^d))\in \Psi_2(C_{2,o}(n)){\setminus\mathcal{S}(n)}$.\medskip

To define $\zeta$, we consider two subcases.   

(I) $\xi_3<c-1$. We define $$\zeta((3,3), (2^{(n-6)/2})):=(\mu(n-18)\cup(9,1),(8))\in \mathcal T'(n),$$ and if 
$(\xi,(c^d))\neq ((3,3), (2^{(n-6)/2}))$, we define
$$\zeta(\xi, (c^d)):=(\xi\setminus (c+1,c+1)\cup \mu(2c+2), (c^d))=:(\nu, (c^d)).$$

Since $(\xi,(c^d))\neq ((3,3), (2^{(n-6)/2}))$, we must have $c>2$. Moreover, $\mu(2c+2)=(c+3, c-1)$ and these are the largest parts in $\nu$.  Since $1, c+1\not \in \nu$ and the largest part in $\nu$ is $c+3$, it follows that $(\nu, (c^d))\not \in \mathcal S(n)$. In fact, $(\nu, (c^d)) \in \mathcal T'(n)$. Furthermore, since $1\not\in\nu$, these pairs do not occur in the images of $\zeta$ in cases (i)--(iii) above.

\medskip 

(II)  If $\xi_3=c-1$, 
 we define 
$$\zeta(\xi, (c^d))=(\xi\setminus (\xi_1, \xi_2, \xi_3) \cup \mu(3c)\cup (1),(c^d))=:(\nu, (c^d)).$$ 

Since $1\not \in \xi$, we have  $c\geq 4$.
Then,  the parts of $\mu(3c)$ are at least $c+1$ and they are the largest parts in $\nu$. Thus $1 \in \nu$ and $c-1\not\in\nu$, and therefore $(\nu, (c^d))\not \in \mathcal S(n)$. Moreover, $(\nu, (c^d))\not \in \mathcal T(n)$. 

\medskip

To finish the proof, for each $n\geq 53$, we display  six  pairs in $B'(n)$ that do not occur in the image of $\zeta$. These pairs belong to $\mathcal T'(n)$. 

If $n\equiv 3\pmod 4$,  two copies of each of the following pairs
   \begin{align*} (\mu(n-15)\cup(5,3,1), (2^3)),\\ (\mu(n-19)\cup (5,3,1), (2^5)),\\ (\mu(n-23)\cup (5,3,1), (2^7)).
   \end{align*}

   If $n\equiv 1\pmod 4$,  two copies of each of the following pairs
     \begin{align*} (\mu(n-17)\cup(7,3,1), (2^3)),\\ (\mu(n-21)\cup (7,3,1), (2^5)),\\ (\mu(n-25)\cup (7,3,1), (2^7)).
   \end{align*}

 If $n$ is even,  two copies of each of the following pairs
   \begin{align*} (\mu(n-14)\cup (7,3), (4)),\\ (\mu(n-20)\cup(9,5), (6)),\\ (\mu(n-18)\cup(11,7), (8)).
   \end{align*}
     Notice that the last three pairs are the only pairs $(\nu, (c^d))$ such that $1\not \in \nu$. Pairs in $\mathcal T'(n)$ with $1\not\in \nu$ occurred in the image of $\zeta$ in case (iv) subcase (I). However, in that case $c+3, c-1$ are the largest parts of $\nu$ while here $c+3, c+1$ are the smallest parts of $\nu$ and, since $n\geq 53$, there are larger parts in $\nu$.

  Thus,  $|B'(n)|\geq |\psi(A(n)|+6$.

\qed

\subsection{Proof of Proposition \ref{prop_bs}}\label{sec_proofpropbs}    
We provide two different proofs of Proposition \ref{prop_bs} below, one of which is combinatorial in nature (see Section \ref{sec_propbscomb}), the other of which is analytic (see Section \ref{sec_propbsanaly}).

\subsubsection{Combinatorial Proof of Proposition \ref{prop_bs}}\label{sec_propbscomb}

As shown in  \cite{WIN5Lehmer}, if $F_2(q) = \sum_{n=1}^\infty b_n q^n$, then $b_n$ equals the number of parts in all partitions of $n$ with distinct odd parts, i.e. $$b_n=\sum_{\lambda\in Q_o(n)}\ell(\lambda).$$

Let $n\geq 9$. We first create a length preserving bijection $\varphi$  from  $\{\lambda \in \mathcal Q_o(n) \mid 1\not \in \lambda\}$  to  $\{\xi\in \mathcal Q_o(n-4)\mid \mbox{if } \ell(\xi)\geq 3, \mbox{ then } \xi_{\ell(\xi)-2}-\xi_{\ell(\xi)-1}>2 \}$.

	Let $\lambda =(\lambda_1, \lambda_2, \ldots, \lambda_{\ell(\lambda)})\in \mathcal Q_o(n)$, $1\not \in \lambda$ and define 
	$$\varphi(\lambda)=\begin{cases} (n-4), & \mbox{ if } \ell(\lambda)=1, \mbox{ i.e., } \lambda=(n),\\ (\lambda_1, \lambda_2, \ldots, \lambda_{\ell(\lambda)-1}-2, \lambda_{\ell(\lambda)}-2), & \mbox{ if }\ell(\lambda)\geq 2. \end{cases}$$
	Clearly, $\ell(\lambda)=\ell(\varphi(\lambda))$. \medskip 
	
	Next, we consider the bijection $\psi$ from $\{\lambda\in \mathcal Q_o(n) \mid 1\in \lambda\}$ to $\{\xi\in \mathcal Q_o(n-1) \mid 1\not \in \xi\}$ given by $\psi(\lambda)=\lambda\setminus(1)$. Clearly, $\ell(\lambda)=\ell(\psi(\lambda))+1$. 
	
	It remains to show that the number of parts equal to $1$ in all partitions in $\mathcal Q_o(n)$ is less than or  equal to   the total number of parts in all partitions in $\mathcal Y:=\{\xi\in \mathcal Q_o(n-1) \mid 1 \in \xi\}\cup\{\xi\in \mathcal Q_o(n-4)\mid  \ell(\xi)\geq 3 \mbox{ and } \xi_{\ell(\xi)-2}-\xi_{\ell(\xi)-1}=2\}$. We create an injection $\Xi$ from $\{\lambda\in \mathcal Q_o(n) \mid 1\in \lambda\}$ to the set of partitions in $\mathcal Y$ with exactly one marked part.

	Let $\lambda\in \mathcal Q_o(n)$ be such that $1\in \lambda$. Since $n\geq 9$, we have $\ell(\lambda)\geq 2$. Let  $a:=\lambda_{\ell(\lambda)-1}$. 	
	
		If $a\geq 9$, $1\not\in \mu(a-1)$, and we define $\Xi(\lambda)=\lambda\setminus (a)\cup \mu(a-1)\in \mathcal Q_o(n-1)$ with part $1$ marked. Note that $\Xi(\lambda)$ has at least three parts and if it has exactly three parts, then the difference between the first and second part is $2$ or $4$.  The marked part of $\Xi(\lambda)$ is the last part.  
		
Next we consider the case when $a=3,5,$ or $7$. Since $n\geq 9$, we have $\ell(\lambda)\geq 3$.

If $\ell(\lambda)\geq 4$,  define $\Xi(\lambda)=\lambda\setminus (\lambda_1, a)\cup(\lambda_1+a-1)\in \mathcal Q_o(n-1)$ with marked  first, second, or third  part according to  $a=3, 5$, or $7$,  respectively.	Note that if $a=7$, the difference between the first and second part in $\Xi(\lambda)$ is at least $8$. Thus, if $\ell(\lambda)=4$ and $a=7$,  then  $\Xi(\lambda)$ has exactly three parts and  the marked part is $1$ but the obtained marked partition  is different from the marked partitions obtained in the case $a\geq 9$. 

If $\ell(\lambda)=3$, then $n$ is odd. We define $\Xi(n-8, 7, 1)=(\overline{n-2}, 1)$, $\Xi(n-6, 5, 1)=(n-2, \overline{1})$ and 
$$\Xi(n-6, 5, 1)=\begin{cases} (\mu(n-5), \overline{1}), & \mbox{ if } n \geq 17, n\equiv 1 \pmod 4,\\ (\mu(n-7), \overline{3}), & \mbox{ if } n \geq 17, n\equiv 3 \pmod 4,\\  (\overline{n-2}, 1), & \mbox{ if } 9\leq n \leq 15. \end{cases}$$
Note that $(n-8,7,1)$ occurs only when $n\geq 17$, and that  $(\mu(n-5), \overline{1}), (\mu(n-7), \overline{3}) \in \mathcal Q_o(n-4)$.
\subsubsection{Analytic Proof of Proposition \ref{prop_bs}}\label{sec_propbsanaly}
 
To prove Proposition \ref{prop_bs}, it suffices to show that 
\begin{align}\label{prop_bs_analytic} (q^4+q-1)F_2(q) \succeq_{S'} 0,\end{align} where $S':=\{0,1,2,3,4,5,6,7,8\}$. 
Indeed, we will prove the stronger result with $S' = \{1,3,4,6,8\}$.
Towards \eqref{prop_bs_analytic}, we establish Lemma \ref{lem_bs_analytic_merged} below, which is stated in terms of the polynomials
\begin{align*}
    f_m(q) := \begin{cases}
    -q + q^2 - q^4 + 2q^5 - q^6, & m=1, \\
    -q^3, & m=2, \\
    -q^5 + q^7 - q^8 + q^9 +2q^{10} + q^{13} - q^{15}, & m=3, \\
    -q^{2m-1}+q^{2m+1}-q^{2m+2}+q^{2m+3}+q^{2m+4}, & m\geq 4.
    \end{cases}
\end{align*}

\begin{lemma}\label{lem_bs_analytic_merged} For each integer $m\geq 1$, we have \begin{align}\label{eqn_ms}(q^4+q-1)q^{2m-1} \mathop{\prod_{\ell=1}^\infty}_{\ell\neq m} (1+q^{2\ell-1}) = 
f_m(q) + g_m(q),\end{align}
 where $g_m(q) \succeq 0,$ and $$g_m(q) = \begin{cases} 
 O(q^7), & m=1, \\
 O(q^5), & m=2, \\
  O(q^{16}), & m=3, \\
 O(q^{2m+6}), & m\geq 4.
 \end{cases}$$
\end{lemma}

Using Lemma \ref{lem_bs_analytic_merged}, we give an analytic proof of Proposition \ref{prop_bs} below.  Following its proof, the remainder of this section is devoted to proving Lemma \ref{lem_bs_analytic_merged}.   

\begin{proof}[Analytic proof of Proposition \ref{prop_bs}] 
By   Lemma \ref{lem_bs_analytic_merged}, we have 
\begin{align}   (q^4+q-1)F_2(q) %&= (q^4+q-1)(-q;q^2)_\infty \sum_{m=1}^3 \frac{q^{2m-1}}{1+q^{2m-1}} \\ & + 
= \sum_{m=1}^\infty (f_m(q)+g_m(q)). \label{eqn_exp}
\end{align} 

We  rewrite $\sum_{m=1}^\infty f_m(q)$ as

\begin{align*}
   &\sum_{j=1}^3f_j(q)+ \sum_{m\geq 4}q^{2m+3}  - \sum_{m\geq 4}(q^{2m-1}+q^{2m+2})+ \sum_{m\geq 4}(q^{2m+1}+q^{2m+4})\\  =& \sum_{j=1}^3f_j(q)+ \sum_{m\geq 4}q^{2m+3} - \sum_{m\geq 4}(q^{2m-1}+q^{2m+2})+ \sum_{m\geq 5}(q^{2m-1}+q^{2m+2})\\=& \sum_{j=1}^3f_j(q)+ \sum_{m\geq 4}q^{2m+3} -q^7-q^{10}\\ =& -q + q^2 -q^3  - q^4 + q^5 - q^6    - q^8 + q^9 +q^{10} +q^{11}+ 2q^{13} + \sum_{m\geq 7}q^{2m+3}.
    \end{align*} 
Since $\sum_{m=1}^\infty g_m(q)\succeq 0$, we obtain the non-negativity of coefficients stated \eqref{prop_bs_analytic}.
 \end{proof}   
 
\begin{proof}[Proof of Lemma \ref{lem_bs_analytic_merged}]    We divide the proof into   cases, depending on  $m$.  \smallskip 

Throughout the proof, we make use of the following calculations. Let $a, i, j$ be positive integers. We express a product $\displaystyle \prod_{k\geq i}(1+q^{2k-1})$ in terms of the smallest, respectively largest, exponent appearing in monomials as \begin{equation}\label{eq_producttoSum}
    1+ \sum_{k\geq i}q^{2k-1}\prod_{\ell>k}(1+q^{2\ell-1}) =  1+ q^{2i-1}+\sum_{k\geq i+1}q^{2k-1}\prod_{i\leq \ell<k}(1+q^{2\ell-1}).
\end{equation} Then, 
\begin{align}\label{prelim-1} & q^{a+2j}\sum_{k\geq i}q^{2k-1}\prod_{\ell>k}(1+q^{2\ell-1})-q^{a}\sum_{k\geq i+j}q^{2k-1}\prod_{\ell>k}(1+q^{2\ell-1}) \\  \nonumber & = q^{a+2j}\sum_{k\geq i}q^{2k-1}\prod_{\ell>k}(1+q^{2\ell-1})-q^{a+2j}\sum_{k\geq i+j}q^{2(k-j)-1}\prod_{\ell>k}(1+q^{2\ell-1})\\ \nonumber & = q^{a+2j}\sum_{k\geq i}q^{2k-1}\prod_{\ell>k}(1+q^{2\ell-1})-q^{a+2j}\sum_{k\geq i}q^{2k-1}\prod_{\ell>k+j}(1+q^{2\ell-1})\succeq 0. \end{align}  Similarly 
\begin{align} \label{prelim-2} & q^{a}\sum_{k\geq i}q^{2k-1}\prod_{\ell>k}(1+q^{2\ell-1})-q^{a+2}\sum_{k\geq i}q^{2k-1}\prod_{\ell>k}(1+q^{2\ell-1}) \\ \nonumber & = q^{a+2i-1}+ q^{a+2i+1}(1+q^{2i-1})+ q^a\sum_{k\geq i+2}q^{2k-1}\prod_{i\leq\ell<k}(1+q^{2\ell-1})\\ \nonumber & \ \ \  - q^{a+2+2i-1}- q^{a+2}\sum_{k\geq i+1}q^{2k-1}\prod_{i\leq\ell<k}(1+q^{2\ell-1})\succeq 0.
\end{align} 
The non-negativity of coefficients follows from the fact that $$q^{a+2}\sum_{k\geq i+1}q^{2k-1}\prod_{i\leq\ell<k}(1+q^{2\ell-1})= q^{a}\sum_{k\geq i+2}q^{2k-1}\prod_{i\leq\ell<k-1}(1+q^{2\ell-1}).$$

We continue with the proof of Lemma \ref{lem_bs_analytic_merged}.

\ \\ {\bf{Case $m\geq 4.$}}
We rewrite the left hand side of \eqref{eqn_ms} as $$P_m(q)Q_m(q),$$ where  
\begin{align*}  P_m(q) &:= (q^4+q-1)q^{2m-1}  (1+q)(1+q^3)(1+q^5) \\
& \  = -q^{2m-1}+q^{2m+1}-q^{2m+2}+q^{2m+3}+q^{2m+4}+2q^{2m+6} +q^{2m+8}\\&\hspace{.2in}+2q^{2m+9}+q^{2m+11} + q^{2m+12}, \\
 Q_m(q)  &:=  \mathop{\prod_{\ell=4}^\infty}_{\ell\neq m} (1+q^{2\ell-1})=1+\mathop{\sum_{k\geq4}}_{k\neq m}q^{2k-1} \mathop{\prod_{\ell>k}}_{\ell\neq m} (1+q^{2\ell-1}).
\end{align*} 
Then,  \begin{align}& \label{neg} g_m(q)=-(q^{2m-1}+q^{2m+2})\mathop{\sum_{k\geq4}}_{k\neq m}q^{2k-1} \mathop{\prod_{\ell>k}}_{\ell\neq m} (1+q^{2\ell-1})\\ \label{pos1} & +(q^{2m+1}+q^{2m+3}+q^{2m+4})\mathop{\sum_{k\geq4}}_{k\neq m}q^{2k-1} \mathop{\prod_{\ell>k}}_{\ell\neq m} (1+q^{2\ell-1})\\ \label{pos2} & +(2q^{2m+6} +q^{2m+8}+2q^{2m+9}+q^{2m+11} + q^{2m+12})Q_m(q).\end{align}
Thus, $g_m(q) = O(q^{2m+6})$.  To show that $g_m(q)\succeq 0$, we show that all terms in \eqref{neg} appear with positive sign in \eqref{pos1} or \eqref{pos2}.  

We first consider the case $m=4$. Then \eqref{neg} equals 
\begin{align*} & -(q^{16}+q^{19}) \prod_{\ell>5} (1+q^{2\ell-1}) -(q^{7}+q^{10}) \sum_{k\geq 6}q^{2k-1}\prod_{\ell>k} (1+q^{2\ell-1})=:-C_1-C_2.  \end{align*} All terms in $C_1$ appear in \eqref{pos2}.  Using \eqref{prelim-1} with  $j=1$, $i=5$, and $a=7, 10$ respectively, terms in $C_2$ cancel with terms in \eqref{pos1}. Thus, $g_4(q)\succeq 0$.

For $m>4$, we rewrite  \eqref{neg}
 separating the terms according to $k=4$, $k\neq 4, m+1$, and $k=m+1$. When $k\neq 4, m+1$, we factor out $q^2$ and shift the index of summation. Thus, \eqref{neg} equals 

\begin{align}& \nonumber -(q^{2m+6}+q^{2m+9}) \mathop{\prod_{\ell>4}}_{\ell\neq m} (1+q^{2\ell-1})\\ \nonumber  & -(q^{2m+1}+q^{2m+4})\mathop{\sum_{k\geq 4}}_{k\neq m-1, m}q^{2k-1} \mathop{\prod_{\ell>k+1}}_{\ell\neq m} (1+q^{2\ell-1})\\ \nonumber & -(q^{4m} \mathop{\prod_{\ell>m+1}} (1+q^{2\ell-1})+ q^{4m+3}\mathop{\prod_{\ell>m+1}} (1+q^{2\ell-1})).
\end{align}
Writing $q^{4m}=q^{2m+3}\cdot q^{2m-3}$ and $q^{4m+3}=q^{2m+6}\cdot q^{2m-3}$,  we see that each term in \eqref{neg} cancels with a corresponding positive term in \eqref{pos1} or \eqref{pos2} (and terms in \eqref{pos1} and \eqref{pos2} are used at most once in this cancellation). Hence,  $g_m(q)\succeq 0$. 
\ \\ {\bf{Case $m=1.$}}
Using \eqref{eq_producttoSum}, we rewrite the left hand side of \eqref{eqn_ms} as \begin{align*}(q^4+q-1)&q \prod_{\ell\geq 2}(1+q^{2\ell-1}) \\ & \hspace{.1in} =q^5+q^2-q+(q^5+q^2-q)\sum_{k\geq 2}q^{2k-1}\prod_{\ell>k}(1+q^{2\ell-1})\\ & \hspace{.1in} = -q+q^2 -q^4+2q^5-q^6 + g_1(q),\end{align*} where 
\begin{align*} g_1(q) = & q^5\sum_{k\geq 2} q^{2k-1}\prod_{\ell>k}(1+q^{2\ell-1})  + q^5\sum_{k\geq 3} q^{2k-1}\prod_{\ell>k}(1+q^{2\ell-1}) \\ +&   q^2\sum_{k\geq 3} q^{2k-1}\prod_{\ell>k}(1+q^{2\ell-1})   - q^4\sum_{k\geq 3} q^{2k-1}\prod_{\ell>k}(1+q^{2\ell-1})\\   -&  q^6 \sum_{k\geq 4} q^{2k-1}\prod_{\ell>k}(1+q^{2\ell-1})  - q\sum_{k\geq 4} q^{2k-1}\prod_{\ell>k}(1+q^{2\ell-1})\\  =: &\,  A_1+A_2+A_3-A_4-A_5-A_6. \end{align*} From this expression, it is clear that $g_1(q)= O(q^7)$.  To show that $g_1(q)\succeq 0$, we first compute $A_1-A_6$.  This equals 
\begin{align*} & q^5\sum_{k\geq 2} q^{2k-1}\prod_{\ell>k}(1+q^{2\ell-1})-q^5\sum_{k\geq 2} q^{2k-1}\prod_{\ell>k+2}(1+q^{2\ell-1})\\
= & \sum_{k\geq 2} (q^{4k+5}+ q^{4k+7}+q^{6k+8})\prod_{\ell>k+2}(1+q^{2\ell-1})=:B_1+ B_2+B_3.
\end{align*}

Next, separating terms by $k$ even and odd respectively, we rewrite $A_5$ as 
$$\sum_{j\geq 2} q^{4j+5}\prod_{\ell>2j}(1+q^{2\ell-1}) + \sum_{j\geq 3} q^{4j+7}\prod_{\ell>2j+1}(1+q^{2\ell-1}).$$ Since $k+2\leq 2k$ if $k\geq 2$, we have $B_1+B_2-A_5\succeq 0$. From \eqref{prelim-2} with $a=2, i=3$, it follows that $A_3-A_4\succeq 0$.  Hence, $g_1(q)\succeq 0$.
\ \\ {\bf{Case $m=2.$}} 
Using \eqref{eq_producttoSum}, we rewrite the left hand side of \eqref{eqn_ms} as $$(q^4+q-1)q^3(1+q) \prod_{\ell\geq 3}(1+q^{2\ell-1})=  -q^3+g_2(q),$$ where $$g_2(q)=(q^5+q^7+q^8)\prod_{\ell\geq 3}(1+q^{2\ell-1})-q^3\sum_{k\geq 3}q^{2k-1} \prod_{\ell>k}(1+q^{2\ell-1})= O(q^5).$$
To show $g_2(q) \succeq 0$, we write \begin{align*} & q^3\sum_{k\geq 3}  q^{2k-1}  \prod_{\ell>k}(1+q^{2\ell-1}) =  q^8  \prod_{\ell>3}(1+q^{2\ell-1})+  q^3  \sum_{k\geq 4}q^{2k-1}\prod_{\ell> k}(1+q^{2\ell-1})\end{align*}
Using \eqref{eq_producttoSum} and \eqref{prelim-1} with $a=3, j=1, i=3$, it follows that $g_2(q) \succeq 0$. 
\ \\ {\bf{Case $m=3.$}} 
Using \eqref{eq_producttoSum}, we rewrite the left hand side of \eqref{eqn_ms} as \begin{align*} (q^4+q-1)&q^5(1+q)(1+q^3) \prod_{\ell\geq 4}(1+q^{2\ell-1}) \\ & = -q^5+q^7-q^8+q^9+2q^{10}+q^{13}-q^{15}+g_3(q),\end{align*} where \begin{align*}g_3(q) =&  q^{12}+q^{15}\\ & +  (-q^5+q^7-q^8+q^9+2q^{10}+q^{12}+q^{13})\sum_{k\geq 4}q^{2k-1}\prod_{\ell> k}(1+q^{2\ell-1})\\  = & (q^9+2q^{10}+q^{12}+q^{13})\sum_{k\geq 4}q^{2k-1}\prod_{\ell> k}(1+q^{2\ell-1})\\ & +(q^{7}+q^{14})\sum_{k\geq 5}q^{2k-1}\prod_{\ell> k}(1+q^{2\ell-1})\\ & -(q^8+q^{15}+ q^{12})\sum_{k\geq 5}q^{2k-1}\prod_{\ell> k}(1+q^{2\ell-1})\\ & -(q^5+q^{14})\sum_{k\geq 6}q^{2k-1}\prod_{\ell> k}(1+q^{2\ell-1})= O(q^{16}).\end{align*}
 Using \eqref{prelim-1} with $a=8, j=1, i=4$ and also with $a=5, j=1, i=5$,  as well as   \eqref{prelim-2} with $a=13, i=5$, we obtain $g_3(q)\succeq 0$.
\end{proof}

\section*{Acknowledgements}  The authors thank the Banff International Research Station (BIRS) and the Women in Numbers 5 (WIN5) Program. The third author is  partially supported  National Science Foundation Grants DMS-1901791 and DMS-2200728. The fifth author 
is partially supported by a FRQNT scholarship by Fonds de Recherche du Qu\'ebec, and an ISM scholarship by Institut des Sciences Math\'ematiques.
\end{document}